\documentclass[11pt,amssymb]{amsart}
\usepackage{amssymb}
\usepackage{amsmath}
\DeclareFontFamily{OT1}{rsfs}{}
\DeclareFontShape{OT1}{rsfs}{n}{it}{<-> rsfs10}{}
\DeclareMathAlphabet{\mathscr}{OT1}{rsfs}{n}{it}

\newcommand{\Z}{{\mathbb Z}}

\newcommand{\C}{{\mathbb C}}
\newcommand{\Q}{{\mathbb Q}}

\newcommand{\R}{{\mathbb R}}

\newcommand{\A}{\mathbb{A}}

\newcommand{\Br}{\mathrm{Br}}

\newcommand{\cK}{\mathscr{K}}

\newcommand{\fX}{\mathfrak{X}}
\newcommand{\cO}{\mathscr{O}}

\newcommand{\Ga}{\mathrm{Gal}}
\newtheorem{thm}{Theorem}[section]
\newtheorem{lemma}[thm]{Lemma}
\newtheorem{prop}[thm]{Proposition}
\newtheorem{cor}[thm]{Corollary}

\newcommand{\cG}{\mathcal{G}}

\newcommand{\gen}{\mathbf{gen}}

\usepackage[all,cmtip]{xy}
\begin{document}

\title[Weakly commensurable groups]{Weakly commensurable groups, with applications to
differential geometry}

\author[G.~Prasad]{Gopal Prasad}

\author[A.~Rapinchuk]{Andrei S. Rapinchuk}

\begin{abstract}
This article is an expanded version of the talk delivered by the
second-named author at the GAAGTA conference. We have included
a discussion of very recent results and conjectures on absolutely
almost simple algebraic groups having the same maximal tori and
finite-dimensional division algebras having the same maximal
subfields; in particular Theorem \ref{T:Finite-NF} contains a
finiteness result in this direction that appears in print for the first
time.
\end{abstract}

\address{Department of Mathematics, University of Michigan, Ann
Arbor, MI 48109}

\email{gprasad@umich.edu}

\address{Department of Mathematics, University of Virginia,
Charlottesville, VA 22904}

\email{asr3x@virginia.edu}

\maketitle

\section{Introduction}\label{S:I}

\noindent {\bf 1.1.} Let $M$ be a Riemannian manifold. In
differential geometry one considers the following sets of data
associated with $M$:

\vskip3mm

$\bullet$ \ \parbox[t]{15.5cm}{$\mathcal{E}(M)$, the {\it spectrum
of the Laplace-Beltrami operator} (for the purposes of this article,
\newline the spectrum is the collection of eigenvalues with their
multiplicities);}

\vskip2mm

$\bullet$ \ \parbox[t]{15.5cm}{$\mathcal{L}(M)$, the {\it length
spectrum}, i.e. the collection of lengths of all closed geodesics in
$M$ {\it with} multiplicities;}

\vskip2mm

$\bullet$ \ \parbox[t]{15.5cm}{$L(M)$, the {\it weak length
spectrum}, i.e. the collection of lengths of all closed geodesics
{\it without} multiplicities.}

\vskip3mm

\noindent (Of course, in order to ensure that the multiplicities in
the definition of $\mathcal{E}(M)$ and $\mathcal{L}(M)$ are finite,
one needs to impose some additional conditions on $M$; however we
will not discuss these technicalities here particularly because in
the case of  compact locally symmetric spaces, which are the most important classes
of manifolds to be considered in this article, problems of this nature do not arise.)
Two Riemannian manifolds $M_1$ and $M_2$ are called
{\it commensurable} if they admit a common finite-sheeted cover $M$,
i.e. if there is a diagram
$$
\xymatrix{ & M \ar[ld]_{\pi_1} \ar[rd]^{\pi_2} & \\ M_1 &  & M_2}
$$
%$$
%\begin{array}{ccccc}
% &  & M & &  \\
% & \pi_1 \swarrow &  & \searrow \pi_2 & \\
%M_1 & & & & M_2
%\end{array},
%$$
in which $M$ is a Riemannian manifold and $\pi_1 , \pi_2$ are
finite-sheeted locally isometric covering maps.

\vskip2mm

We can now formulate the following question that has attracted the
attention of mathematicians working in different areas of
analysis and geometry for quite some time:

\vskip1mm

\noindent {\it Let $M_1$ and $M_2$ be Riemannian manifolds. Are
$M_1$ and $M_2$ necessarily isometric/commensurable if:

\vskip2mm

{\rm (1)} \ $\mathcal{E}(M_1) = \mathcal{E}(M_2)$, i.e. $M_1$ and
$M_2$ are \emph{isospectral};

\vskip2mm

{\rm (2)} \ $\mathcal{L}(M_1) = \mathcal{L}(M_2)$ (or $L(M_1) =
L(M_2)$), i.e. $M_1$ and $M_2$ are \emph{iso-length spectral};

\vskip2mm

{\rm (3)} \ $\Q \cdot L(M_1) = \Q \cdot L(M_2)$, i.e. $M_1$ and
$M_2$ are \emph{length-commensurable}. }

\vskip3mm

\noindent Among conditions (1)-(3), the condition of isospectrality
is definitely the most famous one. In fact, the question of whether
(or when) isospectrality implies isometricity is best known in its
informal formulation due to Mark Kac \cite{Kac} (1966), which is
{\it ``Can you hear the shape of a drum?"} For historical accuracy,
we should point out that the question itself was analyzed in various
forms long before \cite{Kac}, and this analysis had  provided
substantial evidence in favor of the affirmative answer. In
particular, H.~Weyl in 1911 proved a result, which was subsequently
sharpened and generalized by various authors and is currently known
as  {\it Weyl's Law}. It states that if $M$ is an
$n$-dimensional compact Riemannian manifold, and $0 \leqslant \lambda_0
\leqslant \lambda_1 \leqslant \lambda_2 \leqslant \cdots $ is the
sequence of the eigenvalues of the Laplacian of $M$, then
$$
N(\lambda) = \frac{\mathrm{vol}(M)}{(4\pi)^{n/2}
\Gamma\left(\frac{n}{2} + 1\right)} \lambda^n + o(\lambda^n),
$$
where
$$
N(\lambda) = \# \{ j \: \vert \: \sqrt{\lambda_j} \leqslant \lambda
\}
$$
and $\Gamma(s)$ is the $\Gamma$-function (see \cite{ANPS},
\cite{Mue} for a discussion of Weyl's Law, its history and
applications in mathematics and physics). This implies that the
distribution of the eigenvalues alone allows one to recover such
invariants of $M$ as its dimension and volume, and therefore these
invariants are shared by all isospectral Riemannian manifolds.
Moreover, isospectral compact Riemannian manifolds share the heat
kernel invariants (see  \cite{SKY} and references therein). These
powerful analytic techniques and results led one to believe that
isospectral manifolds should indeed be isometric. In 1964, however,
Milnor \cite{Milnor} gave an example of two isospectral, but not
isometric, flat 16-dimensional tori. Then in 1980, M.-F.\,Vign\'eras
\cite{Vig} used the arithmetic of quaternion algebras to construct
nonisometric isospectral Riemann surfaces. A few years later, Sunada
\cite{Sun} found a different method for constructing isospectral
and iso-length spectral, but not isometric, Riemannian manifolds.
Sunada's method relied  on rather simple group-theoretic
properties of the fundamental group, which made it applicable in
many situations. In fact, since its discovery, this method has been
implemented in a variety of situations to produce examples of
nonisometric manifolds for which various geometric invariants are equal (cf., for example,
\cite{LMNR}). Notice, however, that  the
constructions of Vign\'eras and Sunada {\it always} produce
commensurable manifolds. So, it appears that the ``right question"
regarding isospectral manifolds should be {\it whether two
isospectral (compact Riemannian) manifolds are necessarily
commensurable}. While the answer to this question is still negative
in the general case (cf.\,Lubotzky et al.\,\cite{LuSV}), our results
\cite{PR1}, \cite{PR-invol}, \cite{PR-Fields} show that the answer
is in the affirmative for many compact arithmetically defined
locally symmetric spaces (cf.\,Theorem \ref{T:WC12}). Before our
work, such results were available only for  arithmetically defined
hyperbolic 2- and 3-manifolds, cf.\,\cite{CHLR} and \cite{Reid}.

\vskip2mm

Next, we turn to the isometricity question formulated in terms of
condition (2) of iso-length spectrality. It is worth pointing out
that this question is dictated even by  (na\"{\i}ve) geometric
intuition. Indeed, if we take $M_i$ to be the 2-dimensional
Euclidean sphere of radius $r_i$ for $i = 1, 2$, then $L(M_i) = \{
2\pi r_i \}.$ So, in this case the condition of iso-length
spectrality $L(M_1) = L(M_2)$ does imply the isometricity of $M_1$
and $M_2$, and it is natural to ask if this sort of conclusion can
be drawn in a more general situation. Superficially, this question
does not seem to be connected with isospectrality, but in fact using
the trace formula one proves that if $M_1$ and $M_2$ are compact
locally symmetric spaces of nonpositive curvature then
$$
\mathcal{E}(M_1) = \mathcal{E}(M_2) \ \ \Rightarrow \ \ L(M_1) =
L(M_2)
$$
(cf.\,\cite[Theorem 10.1]{PR1}). It follows that nonisometric
isospectral locally symmetric spaces (in particular, those
constructed by Vign\'eras, Sunada and Lubotzky et al.) {\it
automatically} provide examples of nonisometric iso-length spectral
manifolds, and again one should ask about commensurability rather
than isometricity of iso-length spectral manifolds.

While there are important open questions about commensurability
expressed in terms of the classical conditions of isospectrality and
iso-length spectrality (the most famous one being whether two
isospectral Riemann surfaces are necessarily commensurable), it
seems natural to suggest that a systematic study of commensurability
should involve (or even be based upon) conditions that are
invariant under passing to a commensurable manifold. From this
perspective, one needs to point out that conditions (1) and (2), of
isospectrality and iso-length spectrality, respectively, do not
possess this property, while condition (3) of
length-commensurability does. Note that (3) is formulated in terms
of the set $\Q \cdot L(M)$ which is sometimes called the {\it
rational length spectrum}. While this set is not as closely related
to the geometry of $M$ as $L(M)$, it nevertheless has several very
convenient features. First, it indeed does not change if $M$ is
replaced by a commensurable manifold. Second, unlike $L(M)$, which
has been completely identified in very few situations, $\Q \cdot
L(M)$ can be described in more case. Here is one example.

\vskip2mm

\noindent {\bf 1.2. Example.} Let $\mathbb{H} = \{ x + iy \in \C \:
\vert \: y > 0 \}$ be the upper half-plane with the standard
hyperbolic metric $ds^2 = y^{-2}(dx^2 + dy^2)$. It is well-known
that the standard isometric action of $\mathrm{SL}_2(\R)$ on
$\mathbb{H}$ by fractional linear transformations allows us to
identify $\mathbb{H}$ with the symmetric space $\mathrm{SO}_2(\R)
\backslash \mathrm{SL}_2(\R)$. Let $\pi \colon \mathrm{SL}_2(\R) \to
\mathrm{PSL}_2(\R)$ be the canonical projection. Given a discrete
subgroup $\Gamma \subset \mathrm{SL}_2(\R)$ containing $\{ \pm 1 \}$
with torsion-free image $\pi(\Gamma)$, the quotient $M =
\mathbb{H}/\Gamma$ is a Riemann surface. It is well-known that
closed geodesics in $M$ correspond to nontrivial semi-simple
elements in $\Gamma$ (cf.\,\cite[2.1]{PR-Gen}); the precise nature of
this correspondence is not important for us as we only need
information about the length. One shows that if $c_{\gamma}$ is the
closed geodesic in $M$ corresponding to a semi-simple element
$\gamma \in \Gamma$, $\gamma \neq \pm 1$, then its length is given
by
\begin{equation}\label{E:length1}
\ell(c_{\gamma}) = \frac{2}{n_{\gamma}} \cdot \vert \log \vert
t_{\gamma} \vert \vert
\end{equation}
where $t_{\gamma}$ is an eigenvalue of $\gamma$ (note that since
$\pi(\Gamma)$ is discrete and torsion-free, any semi-simple $\gamma
\in \Gamma$ is automatically hyperbolic, i.e. $t_{\gamma} \in \R$),
and $n_{\gamma} \geqslant 1$ is an integer which geometrically is
the winding number and algebraically is the index
$[C_{\Gamma}(\gamma) : \{ \pm 1 \} \cdot \langle \gamma \rangle]$,
where $C_{\Gamma}(\gamma)$ is the centralizer of $\gamma$ in
$\Gamma$ (in other words, $C_{\Gamma}(\gamma) = T \cap \Gamma$ where
$T$ is the maximal $\R$-torus of $\mathrm{SL}_2$ containing
$\gamma$). So, the length spectrum $L(M)$ consists of the values
$\ell(c_{\gamma})$ where $\gamma$ runs over semi-simple elements in
$\Gamma \setminus \{ \pm 1 \}$ with the  property that
$C_{\Gamma}(\gamma) = \{ \pm 1 \} \cdot \langle \gamma \rangle$
(primitive semi-simple elements), while the rational length spectrum
$\Q \cdot L(M)$ is the union of the sets $\Q \cdot \log \vert
t_{\gamma} \vert$, where $\gamma$ runs over all semi-simple $\gamma
\in \Gamma \setminus \{ \pm 1 \}$, and in fact it suffices to take
just one element out of every class of elements having the same
centralizer in $\Gamma$, i.e. one element in $(T \cap \Gamma)
\setminus \{ \pm 1 \}$ for every maximal $\R$-torus $T$ of
$\mathrm{SL}_2$ such that the latter set is non-empty. Now, let us
recall the following example which demonstrates the well-known fact
that the problem of identifying {\it primitive} semi-simple elements
is extremely difficult.

\vskip1mm

Let $D$ be a quaternion division over $\Q$ that splits over $\R$,
and let $\Gamma$ be a torsion-free arithmetic subgroup of $G(\Q)$
where $G = \mathrm{SL}_{1 , D}$. One can view $\Gamma$ as a discrete
subgroup of $G(\R) \simeq {\mathrm{SL}}_2(\R)$. Set $M = \mathbb{H}/\Gamma$. It
is well-known that the maximal subfields of $D$ are of the form $K =
\Q(\sqrt{d})$ with $d$ satisfying $d \notin {\Q^{\times}_p}^2$ for
primes $p$ where $D$ is ramified (thus, the relevant values of $d$
can be easily characterized in terms of congruences). Clearly, the
problem of describing primitive semi-simple elements in $\Gamma$
contains the problem of identifying the fundamental unit
$\varepsilon(d)$ in every such subfield with $d > 0$ (or more
precisely, the smallest unit with norm 1), which is beyond our
reach.
%Now, it follows from the above discussion that describing $L(M)$
%requires the knowledge of primitive semi-simple elements in $\Gamma$
%which contains the problem of finding fundamental units in real
%subfields of $D$, which is beyond our reach.
On the other hand, there is a well-known formula that gives {\it
some} unit in a real quadratic field $\Q(\sqrt{d})$ (assuming that
$d$ is square-free)
\begin{equation}\label{E:unit}
\eta(d) = \prod_{r = 1}^{d-1} \big[\sin \left(\frac{\pi r}{d}\right)\big]^{- \left( \frac{d}{r}
\right)},
\end{equation}
where $\displaystyle \left( \frac{d}{r} \right)$ is the Kronecker
symbol (or the character associated with the quadratic extension),
cf.\,\cite[Ch. V, \S 4, Theorem 2]{BSh}. (Recall that
$\varepsilon(d)$ and $\eta(d)$ are related by the equation $\eta(d)
= \varepsilon(d)^{2 h(d)}$, where $h(d)$ is the class number of
$\Q(\sqrt{d})$, indicating that a systematic description of
$\varepsilon(d)$ for a sufficiently general infinite sequence of
$d$'s is nearly impossible.) So, in this case the rational length
spectrum can be described as the set of all rational multiples of
$\log \eta(d)$, where $\eta(d)$ is given by (\ref{E:unit}) and  $d$
runs through positive square-free integers described by certain
congruences. (It would be interesting to see if one can give a
similar description of the rational length spectrum for other
arithmetically defined locally symmetric spaces - see
\cite[Proposition 8.5]{PR1} regarding the formula for the length of
a closed geodesic in the general case. Obviously, this will require
an intrinsic construction of (sufficiently many) units in a
$\Q$-torus, which has not been offered so far.)
%On the other hand, describing the rational length spectrum $\Q \cdot
%L(M)$ requires finding just one unit of infinite order in each such
%subfield, and such a unit is given by the well-known formula:
%$$
%\eta(d) = \frac{.}{.}
%$$
%So, the rational length spectrum in this case consists of all
%rational multiples of a sequence of numbers described by explicit
%formulas. (It would be interesting to see if one can give a similar
%description of the rational length spectrum in  other situations -
%see \S.. regarding the length formula for arbitrary locally
%symmetric spaces.)

\vskip2mm

This example suggests that at least in some cases the rational
length spectrum $\Q \cdot L(M)$ may be  more tractable than the
length spectrum $L(M)$ or the spectrum $\mathcal{E}(M)$ of the Laplace-Beltrami operator.  But
then the question arises if the rational length spectrum retains
enough information to characterize the commensurability class of
$M$.
%but the question one may have is whether it retains enough
%information to characterize the commensurability class of $M$, at
%least in some situations.
So, we would like to point out that our entire work that resolved
many questions about isospectral and iso-length spectral
arithmetically defined locally symmetric spaces of absolutely simple
real algebraic groups is based on an analysis of the rational
length spectrum and length-commensurability. In fact, with just one
exception, length-commensurability and the (much) stronger condition
of isospectrality lead to the same new results about isospectral
locally symmetric spaces (see Theorem \ref{T:WC12} and the
subsequent discussion). So, we hope that the analysis of the
rational length-spectrum and the associated notion of
length-commensurability will become a standard tool in the
investigation of locally symmetric spaces. Furthermore, the notion
of length-commensurability has an algebraic counterpart, which we
termed {\it weak commensurability} (of Zariski-dense subgroup) - see
\S\ref{S:WC}. What is interesting is that the study of the geometric
problems mentioned earlier has led to a number of algebraic problems
of considerable independent interest such as characterization of
absolutely almost simple algebraic $K$-groups in terms of the
isomorphism classes of their maximal $K$-tori, and in particular,
characterizing finite-dimensional division $K$-algebras in terms of
the isomorphism classes of their maximal subfields. These questions
have been completely resolved for algebraic groups over number
fields and their arithmetic subgroups, and we will review these
results in \S\S\ref{S:Arithm}-\ref{S:Tori}, but remain an area of
active research over general fields, with some important results
obtained very recently (cf.\,\S\ref{S:Tori}). Thus, in broad terms,
our project can be described as the analysis of the consequences of
length-commensurability for locally symmetric spaces and of related
algebraic problems involving classification of algebraic groups
over general (finitely-generated) fields and the investigation of
their Zariski-dense subgroups.

\vskip2mm

\noindent {\bf 1.3. Hyperbolic manifolds.} Cumulatively, our papers
\cite{PR1}, \cite{PR-invol}, \cite{PR-Fields} and the results of
Garibaldi \cite{Gar} and Garibaldi-Rapinchuk \cite{GarR} answer the
key questions about length-commensurable {\it arithmetically
defined} locally symmetric spaces of absolutely simple real algebraic groups of
all types. In particular, we know when length-commensurability
implies commensurability (the answer depends on the Killing-Cartan type of the
group), and that in all cases the arithmetically defined locally
symmetric spaces that are length-commensurable to a given
arithmetically defined locally symmetric space form finitely many
commensurability classes. We will postpone the technical
formulations of these results until \S \ref{S:Geom}, and instead
showcase the consequences of these results for real hyperbolic manifolds.

\addtocounter{thm}{3}

Let $\mathbb{H}^d$ be the real hyperbolic $d$-space. The isometry
group of $\mathbb{H}^d$ is $\mathcal{G} = \mathrm{PO}(d , 1)$, and
by an arithmetic hyperbolic $d$-manifold we mean the quotient $M =
\mathbb{H}^d /\Gamma$ by an {\it arithmetic} subgroup $\Gamma$ of
$\mathcal{G}$ (see \S \ref{S:Arithm} regarding the notion of
arithmeticity). Previously, results about  iso-length spectral
arithmetically defined hyperbolic $d$-manifolds were available only
for $d = 2$ (Reid \cite{Reid}) and $d = 3$ (Reid et al.\,\cite{CHLR}). We obtained the following for length-commensurable
(hence, isospectral) arithmetic hyperbolic manifolds of any
dimension $d \neq 3$.
\begin{thm}
Let $M_1$ and $M_2$ be arithmetically defined hyperbolic
$d$-manifolds.

\vskip1mm

\noindent {\rm (1)} \parbox[t]{16cm}{Suppose $d$ is either even or $\equiv
3(\mathrm{mod} \: 4)$. If $M_1$ and $M_2$ are not commensurable,
then after a possible interchange of $M_1$ and $M_2$, there exists
$\lambda_1 \in L(M_1)$ such that for any $\lambda_2 \in L(M_2)$, the
ratio $\lambda_2/\lambda_1$ is transcendental; in particular, $M_1$
and $M_2$ are not length-commensurable. Thus, in this case
length-commensurability implies commensurability.}

\vskip1mm

\noindent {\rm (2)} \parbox[t]{16cm}{For any $d \equiv
1(\mathrm{mod} \: 4)$ there exist length-commensurable, but not
commensurable, arithmetic hyperbolic $d$-manifolds.}
\end{thm}

\vskip2mm

Furthermore, one can ask about {\it how different are $L(M_1)$ and
$L(M_2)$ (or $\Q \cdot L(M_1)$ and $\Q \cdot (M_2)$) given the fact
that $M_1$ and $M_2$ are \emph{not} length-commensurable}.  For
example, can the symmetric difference $L(M_1) \vartriangle L(M_2)$
be finite? Under some minor additional assumptions, we proved in
\cite{PR-Fields} that if $M_1$ and $M_2$ are non-length
commensurable arithmetically defined hyperbolic $d$-manifolds $(d
\neq 3)$, and $\mathscr{F}_i$ is the subfield of $\R$ generated by
$L(M_i)$ $(i = 1, 2)$,  then the compositum
$\mathscr{F}_1\mathscr{F}_2$ has infinite transcendence degree over
at least one of the fields $\mathscr{F}_1$ or $\mathscr{F}_2$.
(Informally, this means that if $M_1$ and $M_2$ are not
length-commensurable then the sets $L(M_1)$ and $L(M_2)$ are {\it
very} different.) In fact, the same conclusion holds true for
quotients $M_i = \mathbb{H}^{d_i}/\Gamma_i$ $(i = 1, 2)$ by {\it
any} Zariski-dense subgroups $\Gamma_i$ of  $\mathrm{PO}(d_i ,
1)$ if $d_1 \neq d_2$ (assuming that $d_1 , d_2 \neq 3$). We have
similar results for {\it complex} and {\it quaternionic}  hyperbolic
spaces.

\section{Weakly commensurable Zariski-dense subgroups}\label{S:WC}

\noindent {\bf 2.1.} The method developed  for studying the consequences of the
length-commensurability of two locally symmetric spaces is based on
translating the problem into a study of the implications of weak
commensurability of their fundamental groups. To motivate the
formal definition, let us return for a moment to the case of Riemann
surfaces which we considered in Example 1.2. Let $M_1 =
\mathbb{H}^2/\Gamma_1$ and $M_2 = \mathbb{H}^2/\Gamma_2$ be two
Riemann surfaces where $\Gamma_1 , \Gamma_2 \subset {\mathrm{SL}}_2(\R)$ are
discrete subgroups with torsion-free images in ${\mathrm{PSL}}_2(\R)$. For $i =
1, 2$, let $c_{\gamma_i}$ be a closed geodesic in $M_i$
corresponding to a semi-simple element $\gamma_i \in \Gamma_i
\setminus \{ \pm 1 \}$.  Then it follows from (\ref{E:length1}) that
$$
\ell_{\Gamma_1}(\gamma_1) / \ell_{\Gamma_2}(\gamma_2) \: \in \: \Q \
\ \ \Leftrightarrow \ \ \ \exists \:  m , n \: \in \: \mathbb{N} \ \
\text{such \ that} \ \ t^m_{\gamma_1} \: = \: t^n_{\gamma_2},
$$
or equivalently, the subgroups generated by the eigenvalues have
nontrivial intersection. This leads us to the following.

\vskip2mm

\noindent {\bf 2.2. Definition.} Let $G_1 \subset \mathrm{GL}_{N_1}$
and $G_2 \subset \mathrm{GL}_{N_2}$ be two semi-simple algebraic
groups defined over a field $F$ of characteristic zero.

\vskip1mm

\noindent (a) \parbox[t]{16cm}{Semi-simple elements $\gamma_1 \in
G_1(F)$ and $\gamma_2 \in G_2(F)$ are said to be {\it weakly
commensurable} if the subgroups of $\overline{F}^{\times}$ generated
by their eigenvalues intersect nontrivially.}

\vskip2mm

\noindent (b) \parbox[t]{16cm}{(Zariski-dense) subgroups $\Gamma_1
\subset G_1(F)$ and $\Gamma_2 \subset G_2(F)$ are {\it weakly
commensurable} if every semi-simple element $\gamma_1 \in \Gamma_1$
of infinite order is weakly commensurable to some semi-simple
element $\gamma_2 \in \Gamma_2$ of infinite order, and vice versa.}

\vskip3mm

It should be noted that in \cite{PR1} we gave a more technical, but
{\it equivalent}, definition of weakly commensurable elements, viz.
we required the existence of  maximal $F$-tori $T_i$ of $G_i$ for $i
= 1, 2$ such that $\gamma_i \in T_i(F)$ and for some characters
$\chi_i \in X(T_i)$ we have
$$
\chi_1(\gamma_1) = \chi_2(\gamma_2) \neq 1.
$$
This reformulation of (a) demonstrates that the notion of weak
commensurability does not depend on the choice of matrix
realizations of the $G_i$'s (and is in fact more convenient in our
proofs).

\vskip2mm

The above discussion of Riemann surfaces implies that if two Riemann
surfaces $M_1 = \mathbb{H}^2/\Gamma_1$ and $M_2 =
\mathbb{H}^2/\Gamma_2$ are length-commensurable, then the
corresponding fundamental groups $\Gamma_1$ and $\Gamma_2$ are
weakly commensurable. As we will see later, the same conclusion
remains valid for general locally symmetric spaces of finite volume
(cf.\,Theorem \ref{T:6-1}), but now we would like to depart from
geometry and discuss some algebraic aspects of weak
commensurability, along with a few problems of independent interest
that its analysis leads to.

\vskip2mm

From a purely algebraic point of view, the investigation of weakly
commensurable Zariski-dense subgroups fits into the classical
framework of characterizing linear groups in terms of the spectra
(eigenvalues) of its elements. However in the set-up described in
Definition 2.2 it is not obvious at all how one should match the
eigenvalues of (semi-simple) elements $\gamma_1 \in \Gamma_1$ and
$\gamma_2 \in \Gamma_2$ as generally speaking $\gamma_1$ has $N_1$
eigenvalues and $\gamma_2$ has $N_2$. In the theory of complex
representations of finite groups, one combines the eigenvalues of
elements into the character values, and organizes the information
about eigenvalues into the character table which involves {\it all}
representations. This approach appears problematic for Zariski-dense
subgroups of semi-simple algebraic groups as the ambient groups
$G_1$ and $G_2$ have infinitely many inequivalent representations,
so matching somehow their representations and requiring two elements
to have the same eigenvalues in {\it all} respective representations
is not practical, to say the least.
On the other hand, instead of considering all representations, one
could try to match the eigenvalues in a ``canonical'' matrix
realization of the ambient group, but unfortunately it is not clear
which matrix realization should be considered canonical. A
reasonable alternative to these two extreme approaches would be to
match the eigenvalues of $\gamma_1 \in \Gamma_1$ and $\gamma_2 \in
\Gamma_2$ in {\it some} two representations of $G_1$ and $G_2$,
respectively. This, however, actually brings us back to the notion
of weak commensurability. Indeed, given an algebraic $F$-group $G
\subset \mathrm{GL}_N$ and a (semi-simple) element $\gamma \in G(F)$
with eigenvalues $\lambda_1, \ldots , \lambda_N \in
\overline{F}^{\times}$, for any rational representation $\rho \colon
G \to \mathrm{GL}_{N'}$, every eigenvalue of $\rho(\gamma)$ is of
the form $\lambda^{m_1}_1 \cdots \lambda^{m_N}_N$ (where $m_1,
\ldots , m_N$ are some integers), in other words, it is an element
of the subgroup of $\overline{F}^{\times}$ generated by the
eigenvalues of $\gamma$ in the original representation. Conversely,
any element of this subgroup can be realized as an eigenvalue of
$\rho(\gamma)$ in {\it some} rational representation $\rho$. Thus,
in the above notations, semi-simple $\gamma_1 \in \Gamma_1$ and
$\gamma_2 \in \Gamma_2$ are weakly commensurable if one nontrivial
eigenvalue of $\rho_1(\gamma_1)$ in {\it some} rational
representation $\rho_1 \colon G_1 \to \mathrm{GL}_{N'_1}$ equals an
eigenvalue of $\rho_2(\gamma_2)$ in {\it some} rational
representation $\rho_2 \colon G_2 \to \mathrm{GL}_{N'_2}$.
Consequently, weak commensurability provides a way of matching the
eigenvalues of semi-simple elements in $\Gamma_1$ and $\Gamma_2$
that is independent of the choice of the original representations of
$G_1$ and $G_2$. (Using the fact that a finitely generated linear
group over a field of characteristic zero contains a {\it neat}
subgroup of finite index (cf.\,\cite[Theorem 6.11]{Rag-book}), one
shows that if $\Gamma_1$ and $\Gamma_2$ are finitely generated and
$\Delta_i \subset G_i(F)$ is commensurable with $\Gamma_i$ then
$\Gamma_1$ and $\Gamma_2$ are weakly commensurable if and only if
$\Delta_1$ and $\Delta_2$ are weakly commensurable (see Lemma 2.3 of
\cite{PR1}). Consequently, in the analysis of weak commensurability
of finitely generated subgroups $\Gamma_1$ and $\Gamma_2$, one can
assume that the subgroups are neat, and we would like to observe
that in this case the weak commensurability of $\gamma_1 \in
\Gamma_1$ and $\gamma \in \Gamma_2$ implies the weak
commensurability of $\gamma^m_1$ and $\gamma^n_2$ for any nonzero
$m$ and $n$.)

Thus, the relation of weak commensurability of two Zariski-dense
subgroups of semi-simple algebraic groups very loosely corresponds
to the relation between two finite groups under which each column of
the character table of one group contains an element that appears in
the character table for the other group, and vice versa. Clearly,
the latter relation is inconsequential for finite groups, viz.\,it
may hold for infinitely many pairs of nonisomorphic groups, and
therefore does not impose any significant restrictions on the finite
groups at hand.

So, we find it quite remarkable that the weak commensurability of
Zariski-dense subgroups enables one to recover some characteristics
of $\Gamma_1$ and $\Gamma_2$ (and/or $G_1$ and $G_2$) - this is made
possible by the existence in $\Gamma_1$ and $\Gamma_2$ of special
elements called {\it generic elements}, see \S \ref{S:Tori} and
\cite[\S 9]{PR-Gen} . We begin our account of the results for weakly
commensurable Zariski-dense subgroups with the following.

\addtocounter{thm}{2}

\begin{thm}\label{T:WC1}
{\rm (\cite[Theorem 1]{PR1})} Let $G_1$ and $G_2$ be two connected
absolutely almost simple algebraic groups defined over a field $F$
of characteristic zero. Assume that there exist finitely generated
Zariski-dense subgroups $\Gamma_i$ of  $G_i(F)$ which are weakly
commensurable. Then either $G_1$ and $G_2$ are of the same
Killing-Cartan type, or one of them is of type $\textsf{B}_{n}$ and
the other is of type $\textsf{C}_{n}$ for some $n \geqslant 3$.
\end{thm}

By a famous theorem of Tits \cite{Tits-free}, for any semi-simple
group $G$ over a field $F$ of characteristic zero, the group $G(F)$
contains a free Zariski-dense subgroup. So, one or both subgroups in
Theorem \ref{T:WC1} may very well be free, and hence carry no {\it
structural} information about the ambient algebraic group.
Nevertheless, the information about the eigenvalues of  elements
expressed in terms of weak commensurability allows one to see the
type of the ambient algebraic group - we refer to this type of
phenomenon as {\it eigenvalue rigidity}.

There is one more important characteristic that can be seen through
the lense of weak commensurability. Given a Zariski-dense subgroup
$\Gamma$ of  $G(F)$, where $G$ is an absolutely almost simple algebraic group
defined over a field $F$ of characteristic zero, we let $K_{\Gamma}$
denote the subfield of $F$ generated by the traces $\mathrm{Tr} \:
\mathrm{Ad} \: \gamma$ for all $\gamma \in \Gamma$ (the so-called
{\it trace field}). According to a theorem of E.B.~Vinberg
\cite{Vin1}, $K = K_{\Gamma}$ is the minimal field of definition of
$\mathrm{Ad} \: \Gamma$, i.e. the minimal subfield of $F$ such that
one can pick a basis in the Lie algebra of $G$ in which all elements
of ${\mathrm{Ad}}\:\Gamma$ are (simultaneously) represented by matrices with
entries in $K$.

\begin{thm}\label{T:WC2}
{\rm (\cite[Theorem 2]{PR1})} Let $G_1$ and $G_2$ be two connected
absolutely almost simple algebraic groups defined over a field $F$
of characteristic zero. For $i =1,2$, let $\Gamma_i$ be a finitely
generated Zariski-dense subgroup of $G_i(F)$, and $K_{\Gamma_i}$ be
the subfield of $F$ generated by the traces ${\mathrm{Tr}}\,
\mathrm{Ad}\:\gamma$, in the adjoint representation,  of  $\gamma
\in \Gamma_i$. If $\Gamma_1$ and $\Gamma_2$ are weakly
commensurable, then $K_{\Gamma_1} = K_{\Gamma_2}.$
\end{thm}

\vskip2mm

What would be the strongest, hence most desirable,  consequence of
weak commensurability? We recall that two subgroups $\Delta_1$ and
$\Delta_2$ of an abstract group $\Delta$ are called {\it
commensurable} if
$$[\Delta_i : \Delta_1 \cap \Delta_2] < \infty \ \ \text{for} \ \ i = 1,
2.$$ In our set-up of Zariski-dense subgroups $\Gamma_1 \subset
G_1(F)$ and $\Gamma_2 \subset G_2(F)$ in different groups, this
notion needs to be modified as follows. Let $\pi_i \colon G_i \to
\overline{G}_i$ be the isogeny onto the corresponding adjoint group.
We say that $\Gamma_1$ and $\Gamma_2$ are {\it commensurable up to
an $F$-isomorphism between $\overline{G}_1$ and $\overline{G}_2$} if
there exists an $F$-isomorphism $\sigma \colon \overline{G}_1 \to
\overline{G}_2$ such that the subgroups $\sigma(\pi_1(\Gamma_1))$
and $\pi_2(\Gamma_2)$ are commensurable in the usual sense (we note
that the commensurability of locally symmetric spaces is consistent
with this notion of commensurability for the corresponding
fundamental groups). It is easy to see that Zariski-dense subgroups
commensurable up to an isomorphism between the corresponding adjoint
groups are always weakly commensurable. So, the central question is
to determine when the converse is true. As the following example
shows, the desired conclusion may not be valid even if one of the
groups is arithmetic.

\vskip2mm

\noindent {\bf Example 2.5.} (cf.\,\cite[Remark 5.5]{PR1}) Let
$\Gamma$ be a torsion-free Zariski-dense subgroup of $G(F)$. For an
integer $m > 1$, we let $\Gamma^{(m)}$ denote the subgroup generated
by the $m$th powers of  elements of $\Gamma$. Clearly,
$\Gamma^{(m)}$ is weakly commensurable to $\Gamma$ for any $m$. On
the other hand, in many situations, $\Gamma^{(m)}$ is of infinite
index in $\Gamma$ for all sufficiently large $m$. This is, for
example, the case if $\Gamma$ is a nonabelian free group, in
particular, a finite index subgroup of ${\mathrm{SL}}_2(\Z)$. It is also the
case for finite index subgroups of ${\mathrm{SL}}_2(\mathcal{O}_d)$ where
$\mathcal{O}_d$ is the ring of integers in the imaginary quadratic
field $\Q(\sqrt{-d})$, and for cocompact lattices in  semi-simple Lie
groups of $\R$-rank 1. In all these examples, $\Gamma^{(m)}$ (for $m \gg 0$) is
weakly commensurable, but not commensurable to $\Gamma$.

\vskip2mm

This example indicates that an ``ideal" result asserting that the
weak commensurability of $\Gamma_1$ and $\Gamma_2$ implies their
commensurability, generally speaking, is possible only if {\it both}
subgroups $\Gamma_1$ and $\Gamma_2$ are ``large" (e.g., arithmetic
subgroups or at least lattices). Such results  were indeed obtained
in \cite{PR1} (with the help of the results from \cite{Gar} and
\cite{PR-invol}) for $S$-arithmetic groups, and we will review these
results in the next section. For general Zariski-dense subgroups,
one should focus on characterizing the ``minimal" algebraic group
containing the subgroup (in fact, this is precisely the approach
that has led to the definitive results in the arithmetic situation).
More precisely, let $\Gamma \subset G(F)$ be a Zariski-dense
subgroup, and let $K = K_{\Gamma}$ be the corresponding trace field.
Assuming that $G \subset \mathrm{GL}_N$ is adjoint, we know by
Vinberg's theorem that one can pick a basis of the $N$-dimensional
space such that in this basis $\Gamma \subset \mathrm{GL}_N(K)$.
Then the Zariski-closure $\mathscr{G}$ of $\Gamma$ is an algebraic
$K$-group that becomes isomorphic to $G$ over $F$; in other
words, $\mathscr{G}$ is an $F/K$-form of $G$ such that $\Gamma
\subset \mathscr{G}(K)$. Moreover, if $\mathscr{G}'$ is another
$F/K$-form with this property, there exists an $F$-isomorphism
$\mathscr{G} \to \mathscr{G}'$ that induces the ``identity'' map on
$\Gamma$, and then the Zariski-density of $\Gamma$ implies that this
isomorphism is defined over $K$. Thus, the $F/K$-form $\mathscr{G}$ is
uniquely defined. We would now like to formulate the following
finiteness conjecture (which can be compared with the result that
there are only finitely many finite groups with a given character
table).

\vskip2mm

\noindent {\bf Conjecture 2.6.} {\it Let $G_1$ and $G_2$ be
absolutely simple algebraic $F$-groups of adjoint type,
let $\Gamma_1$ be a finitely
generated Zariski-dense subgroup of $G_1(F)$ with trace field  $K =
K_{\Gamma_1}$.  Then there exists a \emph{finite} collection
$\mathscr{G}^{(1)}_2, \ldots , \mathscr{G}^{(r)}_2$ of $F/K$-forms
of $G_2$
%(i.e., each
%$\mathcal{G}^{(i)}_2$ is a $K$-defined algebraic group that becomes
%$F$-isomorphic to $G_2$ after the base change $F/K$)
such that if $\Gamma_2$ is a finitely generated
Zariski-dense subgroup of $G_2(F)$ that is weakly commensurable to $\Gamma_1$,
then it is conjugate to a subgroup of one of the
$\mathscr{G}^{(i)}_2(K)$'s \ $(\subset G_2(F))$.}

\vskip2mm

In \S \ref{S:Finite}, we will present a previously unpublished
result that  implies the truth of this conjecture when $K$ is a
number field. We recall that if $G$ is a simple algebraic $\R$-group
different from $\mathrm{PGL}_2$, then for any lattice $\Gamma$ of
$G(\R)$, the trace field $K_{\Gamma}$ is a number field, so to prove
this conjecture for all lattices in simple groups, it remains to
consider the group $G = \mathrm{PGL}_2$. This has not been done yet, but we
will discuss some related results in this direction in \S
\ref{S:Tori} (specifically, see Theorems \ref{T:Finite1} and
\ref{T:Finite2}). Of course, one is interested not only in
qualitative results in the spirit of Conjecture 2.6, but also in
more quantitative ones asserting that in certain situations $r = 1$,
i.e. a $K$-form is uniquely determined by the weak commensurability
class of a finitely generated Zariski-dense subgroup with trace
field $K$. We refer the reader to Theorem \ref{T:WC4} regarding
results of this kind for arithmetic groups; more general cases
have not been considered so far - see, however, Theorem \ref{T:StabThm}
and Corollary \ref{C:1}.

\section{Results on weak commensurability of $S$-arithmetic
groups}\label{S:Arithm}

\noindent{\bf  3.1. The definition of arithmeticity.} Our results on
weakly commensurable $S$-arithmetic subgroups in absolutely almost
simple groups rely on a specific form of their description, so we
begin with a review of the relevant definitions. Let $G$ be an
algebraic group defined over a number field $K$, and let $S$ be a
finite subset of the set $V^K$ of all places of $K$ containing the
set $V_{\infty}^K$ of archimedean places. Fix a $K$-embedding $G
\subset \mathrm{GL}_N$, and consider the group of $S$-integral
points
$$
G(\cO_K(S)) := G \cap {\mathrm{GL}}_N(\cO_K(S)).
$$
Then, for any field extension $F/K$, the subgroups of $G(F)$ that
are commensurable (in the usual sense)
%\footnote{We recall that two subgroups
%$\mathcal{H}_1$ and $\mathcal{H}_2$ of an abstract group
%$\mathcal{G}$ are called {\it commensurable} if their intersection
%$\mathcal{H}_1 \cap \mathcal{H}_2$ is of finite index in each of the
%subgroups.}
with $G(\cO_K(S))$ are called $S$-{\it arithmetic}, and in the case
where $S = V_{\infty}^K$ simply {\it arithmetic} (note that
$\cO_K(V_{\infty}^K) = \cO_K$, the ring of algebraic integers in
$K$). It is well-known  that the resulting class of $S$-arithmetic
subgroups does not depend on the choice of $K$-embedding $G \subset
\mathrm{GL}_N$ (cf.\,\cite{PlR}). The question, however, is what we
should mean by an arithmetic subgroup of $G(F)$ when $G$ is an
algebraic group defined over a field $F$ of characteristic zero that
is not equipped with a structure of $K$-group over some number field
$K \subset F$. For example, what is an arithmetic subgroup of
$G(\R)$ where $G = \mathrm{SO}_3(f)$ and $f = x^2 + e y^2 - \pi
z^2$? For absolutely almost simple groups the ``right" concept that
we will formalize below is given in terms of forms of $G$ over
the subfields $K \subset F$ that are number fields. In our example,
we can consider the following rational quadratic forms that are
equivalent to $f$ over $\R$:
$$
f_1 = x^2 + y^2 - 3z^2 \ \ \ \text{and} \ \ \ f_2 = x^2 + 2y^2 -
7z^2,
$$
and set $G_i = \mathrm{SO}_3(f_i)$. Then for each $i = 1, 2$, we
have an $\R$-isomorphism $G_i \simeq G$, so the natural arithmetic
subgroup  $G_i(\Z) \subset G_i(\R)$ can be thought of as an
``arithmetic" subgroup of $G(\R)$. Furthermore, one can consider
quadratic forms over other number subfields $K \subset \R$ that
again become equivalent to $f$ over $\R$; for example,
$$
K = \Q(\sqrt{2}) \ \ \ \text{and} \ \ \ f_3 = x^2 + y^2 - \sqrt{2}
z^2.
$$
Then for $G_3 = \mathrm{SO}_3(f_3)$, there is an $\R$-isomorphism
$G_3 \simeq G$ which allows us to view the natural arithmetic
subgroup $G_3(\cO_K) \subset G_3(\R)$, where $\cO_K = \Z[\sqrt{2}]$,
as an ``arithmetic" subgroup of $G(\R)$. One can easily generalize
such constructions from arithmetic to $S$-arithmetic groups by
replacing the rings of integers with the rings of $S$-integers. So,
generally speaking, by an $S$-arithmetic subgroup of $G(\R)$ we mean
a subgroup which is commensurable to one of the subgroups obtained
through this construction for some choice of a number subfield $K
\subset \R$, a finite set $S$ of places of $K$ containing all the
archimedean ones, and a quadratic form $\tilde{f}$  over $K$ that
becomes equivalent to $f$ over $\R$. The technical definition is as
follows.

\vskip2mm

Let $G$ be a connected  absolutely almost simple  algebraic group
defined over a field $F$ of characteristic zero, $\overline{G}$ be
its adjoint group, and $\pi \colon G \to \overline{G}$ be the
natural isogeny. Suppose we are given the following data:

\vskip2mm

$\bullet$ \parbox[t]{15cm}{a {\it number field} $K$ with a {\it
fixed} embedding $K \hookrightarrow F$;}

\vskip1mm

$\bullet$ \parbox[t]{15cm}{an $F/K$-form $\mathscr{G}$ of
$\overline{G}$, which is an algebraic $K$-group such that there
exists an $F$-isomorphism ${}_F \mathscr{G} \simeq \overline{G}$,
where ${}_F\mathscr{G}$ is the group obtained from $\mathscr{G}$ by
extension of scalars from $K$ to $F$;}

\vskip1mm

$\bullet$ \parbox[t]{15cm}{a finite set $S$ of places of $K$
containing $V_{\infty}^K$ but not containing any nonarchimedean
places $v$ such that $\mathscr{G}$ is
$K_v$-anisotropic\footnotemark.}

\footnotetext{We note that if $\mathscr{G}$ is $K_v$-anisotropic
then $\mathscr{G}(\cO_K(S))$ and $\mathscr{G}(\cO_K(S \cup \{v \}))$
are commensurable, and therefore the classes of $S$- and $(S \cup \{
v \})$-arithmetic subgroups coincide. Thus, this assumption on $S$
is necessary if we want to recover it from a given $S$-arithmetic
subgroup.}

\vskip2mm

\noindent We then have an embedding $\iota \colon \mathscr{G}(K)
\hookrightarrow \overline{G}(F)$ which is well-defined up to an
$F$-automorphism of $\overline{G}$ (note that we do {\it not} fix an
isomorphism ${}_F\mathscr{G} \simeq \overline{G}$). A subgroup
$\Gamma$ of $G(F)$ such that $\pi(\Gamma)$ is commensurable with
$\sigma(\iota (\mathscr{G}(\cO_{K}(S))))$, for some $F$-automorphism
$\sigma$ of $\overline{G}$, will be called a $(\mathscr{G} , K ,
S)$-{\it arithmetic subgroup}\footnote{This notion of arithmetic
subgroups coincides with that in Margulis' book \cite{Mar} for
absolutely simple adjoint groups.}, or an $S$-arithmetic subgroup
described in terms of the triple $(\mathscr{G}, K,S)$. As usual,
$(\mathscr{G} , K, V^{K}_{\infty})$-arithmetic subgroups will simply
be called $(\mathscr{G} , K)$-arithmetic.
%
%\vskip2mm
%
%We also need to introduce a more general notion of commensurability.
%The point is that since weak commensurability is defined in terms of
%eigenvalues, a subgroup $\Gamma \subset G(F)$ is weakly
%commensurable with any conjugate subgroup, while the latter may not
%be commensurable with the former in the usual sense. So, to make
%theorems asserting that in certain situations ``{\it weak
%commensurability implies commensurability}" possible (and such
%theorems are in fact one of the goals of our analysis) one
%definitely needs to modify the notion of commensurability. The
%adequate notion (which, in particular, works well in geometric
%applications) is as follows. Let $G_i,$ for $i = 1,\,2,$ be a
%connected absolutely almost simple $F$-group, and let $\pi_i \colon
%G_i \to \overline{G}_i$ be the isogeny onto the corresponding
%adjoint group. We will say that the subgroups $\Gamma_i$ of $G_i(F)$
%are {\it commensurable up to an $F$-isomorphism between
%$\overline{G}_1$ and $\overline{G}_2$} if there exists an
%$F$-isomorphism $\sigma \colon \overline{G}_1 \to \overline{G}_2$
%such that $\sigma(\pi_1(\Gamma_1))$ is commensurable with
%$\pi(\Gamma_2)$ in the usual sense.
The key observation is that the description of $S$-arithmetic
subgroups in terms of triples  is very convenient for determining
when two such subgroups $\Gamma_1 \subset G_1(F)$ and $\Gamma_2
\subset G_2(F)$ are commensurable up to an isomorphism between
$\overline{G}_1$ and $\overline{G}_2$.

\addtocounter{thm}{1}
\begin{prop}\label{P:WC1}
{\rm (\cite[Proposition 2.5]{PR1})} Let $G_1$ and $G_2$ be connected
absolutely almost simple algebraic groups defined over a field $F$
of characteristic zero, and for $i = 1, 2$, let $\Gamma_i$ be a
Zariski-dense $(\mathscr{G}_i, K_i, S_i)$-arithmetic subgroup of
$G_i(F)$. Then $\Gamma_1$ and $\Gamma_2$ are commensurable up to an
$F$-isomorphism between $\overline{G}_1$ and $\overline{G}_2$ if and
only if $K_1 = K_2 =: K$, $S_1 = S_2$, and $\mathscr{G}_1$ and
$\mathscr{G}_2$ are $K$-isomorphic.
\end{prop}

It follows from the above proposition that the arithmetic subgroups
$\Gamma_1$, $\Gamma_2$, and $\Gamma_3$ constructed above, of
$G(\R)$, where $G = \mathrm{SO}_3(f)$, are pairwise
noncommensurable: indeed, $\Gamma_3$, being defined over
$\Q(\sqrt{2})$, cannot possibly be commensurable to $\Gamma_1$ or
$\Gamma_2$ as these two groups are defined over $\Q$; at the same
time, the non-commensurability of $\Gamma_1$ and $\Gamma_2$ is a
consequence of the fact that $\mathrm{SO}_3(f_1)$ and
$\mathrm{SO}_3(f_2)$ are not $\Q$-isomorphic since the quadratic
form $f_1$ is anisotropic over $\Q_3$, and $f_2$ is not.

\vskip5mm

In view of Proposition \ref{P:WC1}, the central question in the
analysis of weak commensurability of $S$-arithmetic subgroups is the
following: {\it Suppose we are given two Zariski-dense
$S$-arithmetic subgroups that are described in terms of triples.
Which components of these triples coincide given the fact that the
subgroups are weakly commensurable?} As the following result
demonstrates, two of these components {\it must} coincide.
\begin{thm}\label{T:WC3}
{\rm (\cite[Theorem 3]{PR1})} Let $G_1$ and $G_2$ be two connected
absolutely almost simple algebraic groups defined over a field $F$
of characteristic zero. If Zariski-dense $(\mathscr{G}_i , K_i ,
S_i)$-arithmetic subgroups $\Gamma_i$ of $G_i(F)$ $(i = 1,\,2)$ are
weakly commensurable, then $K_1 = K_2$ and $S_1 = S_2.$
\end{thm}

\vskip1mm

In general, the forms $\mathscr{G}_1$ and $\mathscr{G}_2$ do not
have to be $K$-isomorphic (see \cite{PR1}, Examples 6.5 and 6.6 as
well as the general construction in \S 9). In the next theorem we
list the cases where it can nevertheless be asserted that
$\mathscr{G}_1$ and $\mathscr{G}_2$ are necessarily $K$-isomorphic,
and then give a general finiteness result for the number of
$K$-isomorphism classes.
\begin{thm}\label{T:WC4}
{\rm (\cite[Theorem 4]{PR1})} Let $G_1$ and $G_2$ be two connected
absolutely almost simple algebraic groups defined over a field $F$
of characteristic zero, of the \emph{same} type different from
$\textsf{A}_{n}$, $\textsf{D}_{2n+1}$, with $n > 1$, or
$\textsf{E}_6$. If for $i = 1,\,2 $, $G_i(F)$  contain Zariski-dense
$(\mathscr{G}_i , K , S)$-arithmetic subgroup $\Gamma_i$ which are
weakly commensurable to each other, then $\mathscr{G}_1 \simeq
\mathscr{G}_2$ over $K,$ and hence $\Gamma_1$ and $\Gamma_2$ are
commensurable up to an $F$-isomorphism between $\overline{G}_1$ and
$\overline{G}_2.$
\end{thm}

In this theorem, type $\textsf{D}_{2n}$ $(n \geqslant 2)$ required
special consideration. The case $n > 2$ was settled in
\cite{PR-invol} using the techniques of \cite{PR1} in conjunction
with  results on embeddings of fields with involutive automorphisms
into simple algebras with involution. The remaining case of type
$\textsf{D}_4$ was treated by Skip Garibaldi \cite{Gar}, whose
argument actually applies to all $n$ and explains the result from
the perspective of Galois cohomology, providing thereby a
cohomological insight (based on the notion of Tits algebras) into
the difference between the types $\textsf{D}_{2n}$ and
$\textsf{D}_{2n + 1}$. We note that the types excluded in the
theorem are precisely the types for which the automorphism $\alpha
\mapsto -\alpha$ of the corresponding root system is {\it not} in
the Weyl group. More importantly, all these types are honest
exceptions to the theorem -- a general Galois-cohomological
construction of weakly commensurable, but not commensurable,
Zariski-dense $S$-arithmetic subgroups for all of these types is
given in \cite[\S 9]{PR1}.

\vskip2mm

\begin{thm}\label{T:WC5}
{\rm (\cite[Theorem 5]{PR1})} Let $G_1$ and $G_2$ be two connected
absolutely almost simple groups defined over a field $F$ of
characteristic zero. Let $\Gamma_1$ be a Zariski-dense
$(\mathscr{G}_1 , K , S)$-arithmetic subgroup of $G_1(F).$ Then the
set of $K$-isomorphism classes of $K$-forms $\mathscr{G}_2$ of
$\overline{G}_2$ such that $G_2(F)$ contains a Zariski-dense
$(\mathscr{G}_2 , K , S)$-arithmetic subgroup weakly commensurable
to $\Gamma_1$ is finite.

In other words, the set of all Zariski-dense $(K , S)$-arithmetic
subgroups of  $G_2(F)$ which are weakly commensurable to a given
Zariski-dense $(K , S)$-arithmetic subgroup is a union of finitely
many commensurability classes.
\end{thm}

\vskip2mm

A noteworthy fact about weak commensurability is that it has the
following implication for the existence of unipotent elements in
arithmetic subgroups (even though it is formulated entirely in terms
of semi-simple ones). We recall that a semi-simple $K$-group is
called $K$-{\it isotropic} if $\mathrm{rk}_K\: G > 0$; in
characteristic zero, this is equivalent to the existence of
nontrivial unipotent elements in $G(K)$. Moreover, if $K$ is a
number field then $G$ is $K$-isotropic if and only if every
$S$-arithmetic subgroup contains unipotent elements, for any $S$.
\begin{thm}\label{T:WC6}
{\rm (\cite[Theorem 6]{PR1})} Let $G_1$  and $G_2$ be two connected
absolutely almost simple algebraic groups defined over a field $F$
of characteristic zero. For $i = 1,2$, let $\Gamma_i$ be a
Zariski-dense $(\mathscr{G}_i ,K , S)$-arithmetic subgroup of
$G_i(F)$. If $\Gamma_1$ and $\Gamma_2$ are weakly commensurable then
$\mathrm{rk}_K\mathscr{G}_1 =\mathrm{rk}_K\mathscr{G}_2$; in
particular, if $\mathscr{G}_1$ is $K$-isotropic, then so is
$\mathscr{G}_2$.
\end{thm}

We note that in \cite[\S 7]{PR1} we prove a somewhat more precise
result, viz. that if $G_1$ and $G_2$ are of the same type, then the
Tits indices of $\mathscr{G}_1/K$ and $\mathscr{G}_2/K$ are
isomorphic, but we will not get into these technical details here.

\vskip2mm

The following result asserts that a lattice{\footnote{ A discrete
subgroup $\Gamma$ of a locally compact topological group $\cG$ is
said to be a lattice in $\cG$ if $\cG/\Gamma$ carries a {\it finite}
$\cG$-invariant Borel measure.}} which is weakly commensurable with
an $S$-arithmetic group is itself $S$-arithmetic.
\begin{thm}\label{T:WC7}
{\rm (\cite[Theorem 7]{PR1})} Let $G_1$ and $G_2$ be two connected
absolutely almost simple algebraic groups defined over a nondiscrete
locally compact field $F$ of characteristic zero, and for $i =1,
\,2$, let $\Gamma_i$ be a Zariski-dense lattice in $G_i(F).$ Assume
that $\Gamma_1$ is a $(K , S)$-arithmetic subgroup of $G_1(F)$. If
$\Gamma_1$ and $\Gamma_2$ are weakly commensurable, then $\Gamma_2$
is a $(K , S)$-arithmetic subgroup of $G_2(F)$.
\end{thm}

\vskip2mm

According to Theorem \ref{T:WC1}, if $G_1$ and $G_2$ contain weakly
commensurable finitely generated Zariski-dense subgroups then either
the groups are of the same Killing-Cartan type, or one of them is of
type $\textsf{B}_n$ and the other is of type $\textsf{C}_n$ for some
$n \geqslant 3$. Weakly commensurable $S$-arithmetic subgroups in
the first case were analyzed in Theorem \ref{T:WC4} (see also the
discussion thereafter). We conclude this section with a recent
result of Skip Garibaldi and the second-named author \cite{GarR}
which gives a criterion for two Zariski-dense $S$-arithmetic
subgroups in the groups of type $\textsf{B}_n$ and $\textsf{C}_n$ to
be weakly commensurable. To formulate the result we need the
following definition. Let $\mathscr{G}_1$ and $\mathscr{G}_2$ be
absolutely almost simple algebraic groups of types $\textsf{B}_n$
and $\textsf{C}_n$ with $n \geqslant 2$, respectively, over a number
field $K$. We say that $\mathscr{G}_1$ and $\mathscr{G}_2$ are {\it
twins} (over $K$) if for each $v \in V^K$, either both groups are
split or both are anisotropic over the completion $K_v$. (We note
that since groups of these types cannot be anisotropic over $K_v$
when $v$ is nonarchimedean, our condition effectively says that
$\mathscr{G}_1$ and $\mathscr{G}_2$ must be $K_v$-split for {\it
all} nonarchimedean $v$.)
\begin{thm}\label{T:BC77}
{\rm (\cite[Theorem 1.2]{GarR})} Let $G_1$ and $G_2$ be absolutely
almost simple algebraic groups over a field $F$ of characteristic
zero of  Killing-Cartan types $\textsf{B}_n$ and $\textsf{C}_n$
$(n \geqslant 3)$ respectively, and let $\Gamma_i$ be a
Zariski-dense $(\mathscr{G}_i, K, S)$-arithmetic subgroup of
$G_i(F)$ for $i = 1, 2$. Then $\Gamma_1$ and $\Gamma_2$ are weakly
commensurable if and only if the groups $\mathscr{G}_1$ and
$\mathscr{G}_2$ are twins.
\end{thm}

\vskip.5mm

(We recall that according to Theorem \ref{T:WC3}, if Zariski-dense
$(\mathscr{G}_1, K_1, S_1)$- and $(\mathscr{G}_2, K_2,
S_2)$-arithmetic subgroups are weakly commensurable then necessarily
$K_1 = K_2$ and $S_1= S_2$, so Theorem \ref{T:BC77} in fact treats
the most general situation.)

\vskip3mm

\section{Absolutely almost simple algebraic groups having the same
maximal tori}\label{S:Tori}

\vskip2mm

\noindent {\bf 4.1.} The analysis of weak commensurability leads to,
and also depends on, problems of an algebraic nature
that in  broad terms can be described as {\it characterizing
absolutely almost simple algebraic groups over a given (nice) field
$K$ having the same isomorphism/isogeny classes of maximal
$K$-tori}. While these problems are not new (for example, in the
context of finite-dimensional central simple algebras they can be
traced back to such classical algebraic results as Amitsur's Theorem
\cite{Ami} on generic splitting fields - cf.\,\S 4.4 below), there
has been a noticeable resurgence of interest in them in recent
years. One should mention \cite{Garge} and \cite{Kar} where some
aspects of the problem were considered over local and global fields;
the local-global principles for embedding tori into absolutely
almost simple algebraic groups as maximal tori (in particular, for
embedding commutative \'etale algebras with involutive automorphisms
into simple algebras with involution) have been analyzed in \cite{Bayer},
\cite{Gar}, \cite{Lee}, \cite{PR-invol}; some number-theoretic
applications have been given in \cite{Fiori}. In this section, we will
focus primarily on those aspects of the problem that are related to
the study of weak commensurability and particularly to Conjecture
2.6. The most recent results here analyze division algebras having
the same maximal subfields and/or the same splitting fields
\cite{CRR1}, \cite{CRR2}, \cite{GS}, \cite{KMcK}, \cite{RR}. These
results provide, in particular,  substantial supporting evidence for
the Finiteness Conjecture 4.12  about absolutely almost simple
algebraic $K$-groups having the same isomorphism classes of maximal
tori, and hence also for Conjecture 2.6. We will return to  weak
commensurability in the next section and present a result that
indicates a unified approach to both conjectures 2.6 and 4.12 (cf.\,\S \ref{S:Finite}).

\vskip2mm

\noindent {\bf 4.2. Generic elements and the Isogeny Theorem.} We
begin by describing in more precise terms the connection between
weak commensurability and study of  absolutely almost
simple algebraic groups having the same isomorphism classes of
maximal tori. This connection is based on the Isogeny Theorem (see
below) to formulate which we need to recall the notion of {\it
generic tori} and {\it generic elements}.

Let $G$ be a connected absolutely almost simple algebraic group defined over an
infinite field $K$. Fix a maximal $K$-torus $T$ of $G$, and, as
usual, let $\Phi = \Phi(G , T)$ denote the corresponding root
system, and let $W(G , T)$ be its Weyl group. Furthermore, we let
$K_T$ denote the (minimal) splitting field of $T$ in a fixed
separable closure $\overline{K}$ of $K$. Then the natural action of
the Galois group $\Ga(K_T/K)$ on the character group $X(T)$ of $T$
induces an injective homomorphism
$$
\theta_T \colon \Ga(K_T/K) \to \mathrm{Aut}(\Phi(G , T)).
$$
We say that $T$ is {\it generic} (over $K$) if
\begin{equation}\label{E:Gen1}
\theta_T(\Ga(K_T/K)) \supset W(G , T).
\end{equation}
(Note that such a torus is automatically $K$-{\it irreducible}, i.e.
it does not contain proper $K$-subtori.) For example, any maximal
$K$-torus of $G = \mathrm{SL}_n/K$ is of the form $T =
\mathrm{R}^{(1)}_{E/K}(\mathrm{GL}_1)$ for some $n$-dimensional
commutative  \'etale $K$-algebra $E$. Then such a torus is generic
over $K$ if and only if $E$ is a separable field extension of $K$
and the Galois group of the normal closure $L$ of $E$ over $K$ is
isomorphic to the symmetric group $S_n$.
Furthermore, a regular semi-simple element $g \in G(K)$ is called
{\it generic} (over $K$) if the $K$-torus $T = Z_G(g)^{\circ}$
(the identity component of the centralizer $Z_G(g)$ of $g$ in $G$)
is generic (over $K$) in the sense defined above. We are
now in a position to formulate a result that enables one to pass
from the weak commensurability of two generic elements to an
isogeny, and in most cases even to an isomorphism, of the ambient
tori. \addtocounter{thm}{2}
\begin{thm}\label{T:Gen4}
{\rm (Isogeny Theorem, \cite[Theorem 4.2]{PR1})} Let $G_1$ and $G_2$
be two connected absolutely almost simple algebraic groups defined
over an infinite field $K$, and let $L_i$ be the minimal Galois
extension of $K$ over which $G_i$ becomes an inner form of a split
group. Suppose that for $i = 1, 2$, we are given a semi-simple
element $\gamma_i \in G_i(K)$ contained in a maximal $K$-torus $T_i$
of $G_i$. Assume that (i) $G_1$ and $G_2$ are either of the same
Killing-Cartan type, or one of them is of type $\textsf{B}_{n}$ and
the other is of type $\textsf{C}_{n}$, (ii) $\gamma_1$ has infinite
order, (iii) $T_1$ is $K$-irreducible, and (iv) $\gamma_1$ and
$\gamma_2$ are weakly commensurable. Then

\vskip2mm

\noindent {\rm (1)} \parbox[t]{16cm}{there exists a $K$-isogeny $\pi
\colon T_2 \to T_1$ which carries $\gamma^{m_2}_2$ to
$\gamma^{m_1}_1$ for some integers $m_1 , m_2 \geqslant 1$;}

\vskip1mm

\noindent {\rm (2)} \parbox[t]{16cm}{if $L_1 = L_2 =: L$ and
$\theta_{T_1}(\Ga(L_{T_1}/L)) \supset W(G_1 , T_1)$, then $\pi^*
\colon X(T_1) \otimes_{\Z} \Q \to X(T_2) \otimes_{\Z} \Q$ has the
property that $\pi^*(\Q \cdot \Phi(G_1 , T_1)) = \Q \cdot \Phi(G_2 ,
T_2)$. Moreover, if $G_1$ and $G_2$ are of the same Killing-Cartan
type different from $\textsf{B}_2 = \textsf{C}_2$, $\textsf{F}_4$ or
$\textsf{G}_2$, then a suitable rational multiple of $\pi^*$  maps
$\Phi(G_1 , T_1)$ onto $\Phi(G_2 , T_2)$.}
%and if $G_1$ is of type
%$\textsf{B}_{n}$ and $G_2$ is of type $\textsf{C}_{n}$, with $n >
%2$, then a suitable rational multiple $\lambda$ of $\pi^*$ takes the
%long roots in $\Phi(G_1 , T_1)$ to the short roots in $\Phi(G_2 ,
%T_2)$ while $2\lambda$ takes the short roots in $\Phi(G_1 , T_1)$ to
%the long roots in $\Phi(G_2 , T_2)$.}
\end{thm}

It follows that in the situations where $\pi^*$ can be, and has
been, scaled so that $\pi^*(\Phi(G_1 , T_1)) = \Phi(G_2 , T_2)$, the
isogeny $\pi$ induces $K$-isomorphisms $\widetilde{\pi} \colon
\widetilde{T}_2 \to \widetilde{T}_1$ and $\overline{\pi} \colon
\overline{T}_2 \to \overline{T}_1$ between the corresponding tori in
the simply connected and adjoint groups $\widetilde{G}_i$ and
$\overline{G}_i$.
%respectively, that extend to
%$\overline{K}$-isomorphisms $\widetilde{G}_2 \to \widetilde{G}_1$
%and $\overline{G}_2 \to \overline{G}_1$.
Thus, in most situations, the fact that Zariski-dense torsion-free
subgroups $\Gamma_1 \subset G_1(K)$ and $\Gamma_2 \subset G_2(K)$
are weakly commensurable implies (under some minor technical
assumptions) that $G_1$ and $G_2$ have the same $K$-isogeny classes
(and under some minor additional assumptions -- even the same
$K$-isomorphism classes) of generic maximal  $K$-tori that
nontrivially intersect $\Gamma_1$ and $\Gamma_2$, respectively.
Since over finitely generated fields generic tori, and also generic
elements in a given (finitely generated) Zariski-dense subgroup,
exist in abundance (cf.\,\cite{PR-Reg}, and also \cite[\S
9]{PR-Gen}), this relates $G_1$ and $G_2$ in a significant way and
leads to important results (cf., for example, Theorem
\ref{T:Finite-NF}). So, while the problem of understanding algebraic
groups with the same isomorphism/isogeny classes is not completely
equivalent to the investigation of weak commensurability of
Zariski-dense subgroups (for one thing, not every maximal torus
necessarily intersects a given Zariski-dense subgroup), in practice
it does capture most intricacies of the latter, and in fact the
connection between the problems goes both ways. The next theorem
(cf.\,\cite[Theorem 7.5]{PR1} and \cite[Proposition 1.3]{GarR}),
which is a consequence of the results on weak commensurability
(cf.\,\S \ref{S:Arithm}), illustrates this point.
%This suggest that the fundamental question of what algebraic groups
%can contain weakly commensurable heavily depends on (although is not
%completely equivalent to) the understanding of algebraic groups
%having the same isomorphism/isogeny classes of maximal tori.
%In fact, the results on weakly commensurable arithmetic groups (cf.
%\S ..) lead to the following fact about the groups over number
%fields with the same tori
%The analysis of weak commensurability is related to another natural
%problem in the theory of algebraic groups of characterizing
%absolutely almost simple $K$-groups having the same
%isomorphism/isogeny classes of maximal $K$-tori -- the exact nature
%of this connection will be clarified in Theorem \ref{T:Gen4} and the
%subsequent discussion.  Some aspects of this problem over local and
%global fields were considered in \cite{Garge} and \cite{Kar}.
%Another direction of research, which has already generated a number
%of results (cf.\,\cite{Bayer}, \cite{Gar}, \cite{Lee},
%\cite{PR-invol}) is the investigation of local-global principles for
%embedding tori into absolutely almost simple algebraic groups as
%maximal tori (in particular, for embedding of  commutative \'etale
%algebras with involutive automorphisms into simple algebras with
%involution); some number-theoretic applications of these results can
%be found, for example, in \cite{Fiori}. A detailed discussion of
%these issues would be an independent undertaking, so we will limit
%ourselves here to the following theorem
%(cf.\,\cite[Theorem 7.5]{PR1} and \cite[Proposition 1.3]{GarR}).
\begin{thm}\label{T:WC13}
{\rm (1)} Let $G_1$ and $G_2$ be connected absolutely almost simple
algebraic groups defined over a number field $K$, and let $L_i$ be
the smallest Galois extension of $K$ over which $G_i$ becomes an
inner form of a split group. If $G_1$ and $G_2$ have the same
$K$-isogeny classes of maximal $K$-tori then either $G_1$ and $G_2$
are of the same Killing-Cartan type, or one of them is of type
$\textsf{B}_{n}$ and the other is of type $\textsf{C}_{n}$, and
moreover, $L_1 = L_2$.

\vskip2mm

\noindent {\rm (2)} \parbox[t]{16cm}{Fix an absolutely almost simple
$K$-group $G$. Then the set of isomorphism classes of all absolutely
almost simple $K$-groups $G'$ having the same $K$-isogeny classes of
maximal $K$-tori is finite.}

\vskip2mm

\noindent {\rm (3)} \parbox[t]{16cm}{Fix an absolutely almost simple
simply connected $K$-group $G$ whose Killing-Cartan type is
different from $\textsf{A}_{n}$, $\textsf{D}_{2n+1}$ $(n > 1)$ or
$\textsf{E}_6$. Then any $K$-form $G'$ of $G$ (in other words, any
absolutely almost simple simply connected $K$-group $G'$ of the
\emph{same} type as $G$) that has the same $K$-isogeny classes of
maximal $K$-tori as $G$, is isomorphic to $G$.}
\end{thm}

The construction described in \cite[\S 9]{PR1} shows that the types
excluded in (3) are honest exceptions, i.e., for each of those types
one can construct non-isomorphic absolutely almost simple simply
connected $K$-groups $G_1$ and $G_2$ of this type over a number
field $K$ that have the same isomorphism classes of maximal
$K$-tori.

\vskip1mm

The situation where one of the groups is of type $\textsf{B}_n$ and
the other is of type $\textsf{C}_n$ with $n \geqslant 3$ was
analyzed in \cite{GarR}.
\begin{thm}\label{T:BC88}
{\rm (\cite[Theorem 1.4]{GarR})} Let $G_1$ and $G_2$ be absolutely
almost simple algebraic groups over a number field $K$ of types
$\textsf{B}_n$ and $\textsf{C}_n$ respectively for some $n \geqslant
3$.

\vskip1mm

\noindent {\rm (1)} \parbox[t]{15cm}{The groups $G_1$ and $G_2$ have
the same \emph{isogeny} classes of maximal $K$-tori if and only if
they are twins\footnotemark.}

\vskip1mm

\noindent {\rm (2)} \parbox[t]{15cm}{The groups $G_1$ and $G_2$ have
the same \emph{isomorphism} classes of maximal $K$-tori if and only
if they are twins, $G_1$ is adjoint and $G_2$ is simply connected.}
\end{thm}

\footnotetext{See the definition of twins prior to the statement of
Theorem \ref{T:BC77}.}

\vskip2mm

\addtocounter{thm}{1}

\noindent {\bf 4.6. Division algebras with the same maximal
subfields.} As we already mentioned, questions related to the
problem of characterizing absolutely almost simple algebraic
$K$-groups by the isomorphism/isogeny classes of their maximal
$K$-tori were in fact raised and investigated a long time ago,
particularly in the context of finite-dimensional central simple
algebras.
%Of course, the question about determining absolutely almost simple
%algebraic $K$-groups by their maximal $K$-tori makes sense over
%general fields. As we already mentioned, questions of this nature
%were in fact raised a long time ago, particularly in the context of
%finite-dimensional central simple algebras.
We recall that given a central simple algebra $A$ of degree $n$
(i.e., of dimension $n^2$) over a field $K$, a field extension $F/K$
is called a {\it splitting field} if $A \otimes_K F \simeq M_n(F)$
as $F$-algebras; furthermore, if $A$ is a division algebra then the
splitting fields of degree $n$ over $K$ are precisely the maximal
subfields of $A$. It is well-known that the splitting fields/maximal
subfields of a central simple $K$-algebra $A$ play a huge role in
the analysis of its structure (cf., for example, \cite{GiSz}), which
suggests the question: {\it to what extent do these fields
actually determine $A$?} The answer to this question in the
situation where one considers {\it all} splitting fields is given by
the famous theorem of Amitsur \cite{Ami}: {\it Let $A_1$ and $A_2$
be finite-dimensional central simple algebras over a field $K$.
Assume that a field extension $F/K$ splits $A_1$ if and only if it
splits $A_2$. Then the classes $[A_1]$ and $[A_2]$ in the Brauer
group $\Br(K)$ generate the same subgroup: $\langle [A_1] \rangle =
\langle [A_2] \rangle$} (the converse is obvious). The proof of
Amitsur's Theorem (cf.\,\cite{Ami},\,\cite[Ch. 5]{GiSz}) uses {\it
generic splitting fields} which are {\it infinite} extensions of
$K$. At the same time, it is important to point out that the
situation changes dramatically if instead of all splitting fields
one considers only finite-dimensional ones or just maximal
subfields.

\vskip2mm

\noindent {\bf Example 4.7.} Fix $r \geqslant 2$, and pick $r$
distinct primes $p_1, \ldots , p_r$. Let $\varepsilon =
(\varepsilon_1, \ldots , \varepsilon_r)$ be any $r$-tuple with
$\varepsilon_i = \pm 1$ such that  $\sum_{i = 1}^r \varepsilon_i
\equiv 0(\mathrm{mod} \: 3)$. By class field theory (the
Albert-Brauer-Hasse-Noether Theorem - to be referred to as (ABHN) in
the sequel, cf.\,\cite[Ch.\,VII, 9.6]{ANT}, \cite[18.4]{Pierce}, and
also \cite{Roq} for a historical perspective), corresponding to
$\varepsilon$, we have a central cubic division algebra
$D(\varepsilon)$ over $\Q$ with the following local invariants
(considered as elements of $\Q/\Z$):
$$ \mathrm{inv}_p \:
D(\varepsilon) \ = \ \left\{ \begin{array}{ccl} \displaystyle
\frac{\varepsilon_i}{3} & , & p = p_i \ \ \text{for} \ i = 1, \ldots
, r; \\  0 & , & p \notin \{p_1, \ldots , p_r\} \ \
(\text{including} \ p = \infty)
\end{array} \right.
$$
Then for any two $r$-tuples $\varepsilon'$ and $\varepsilon''$ as
above, the algebras $D(\varepsilon')$ and $D(\varepsilon'')$ have
the same finite-dimensional splitting fields, hence the same maximal
subfields (cf.\,\cite[18.4, Corollary b]{Pierce}), and are
non-isomorphic if $\varepsilon' \neq \varepsilon''$. Obviously, the
number of admissible $r$-tuples $\varepsilon$ grows with $r$, so
this method enables one to construct an {\it arbitrarily large} (but
finite) number of pairwise nonisomorphic cubic division algebras
over $\Q$ having the same maximal subfields.

\vskip2mm

A similar construction can be implemented for division algebras
of any degree $d > 2$. On the other hand, it follows from
(ABHN) that any two quaternion division algebras over a number field
$K$ with the same quadratic subfields are necessarily isomorphic.
This suggests the following question:

\vskip2mm

\begin{center}

$(*)$ \parbox[t]{14cm}{\it What can one say about two
finite-dimensional central division algebras $D_1$ and $D_2$ over a
field $K$ given the fact that they have the same (isomorphism
classes of) maximal subfields?}

\end{center}

\vskip2mm

(We say that central division $K$-algebras $D_1$ and $D_2$ have the
same isomorphism classes of maximal subfields if they have the same
degree $n$ and a degree $n$ field extension $F/K$ admits a
$K$-embedding $F \hookrightarrow D_1$ if and only if it admits a
$K$-embedding $F \hookrightarrow D_2$.)

\vskip1mm

It should be noted that  $(*)$ is closely related (although not
equivalent) to the question of understanding the relationship
between $D_1$ and $D_2$ when the groups $G_1 = \mathrm{SL}_{1 ,
D_1}$ and $G_2 = \mathrm{SL}_{1 , D_2}$ have the same isomorphism
classes of maximal $K$-tori, and we will comment on this a bit later
(cf.\,Theorem \ref{T:Finite2} and the discussion thereafter). Our
next immediate goal, however, is to present some recent results on
$(*)$, for which we need the following definition.

\vskip2mm

\addtocounter{thm}{2}

\noindent {\bf Definition 4.8.} Let $D$ be a central division
$K$-algebra of degree $n$. The {\it genus} $\gen(D)$ is the set of
all classes $[D'] \in \Br(K)$ represented by central division
$K$-algebras $D'$ having the same maximal subfields as $D$.

\vskip2mm

The following basic questions about the genus represent two aspects
of the general question $(*)$.

\vskip2mm

\noindent {\bf Question 1.} {\it When does $\gen(D)$ reduce to a
single class?}

\vskip1mm

(This is another way of asking whether $D$ is  determined uniquely up
to isomorphism by its maximal subfields.)

\vskip2mm

\noindent {\bf Question 2.} {\it When is $\gen(D)$ finite?}

\vskip2mm

Regarding Question 1, we note that $\vert \gen(D) \vert = 1$ is
possible only if $D$ has exponent 2 in the Brauer group. Indeed, the
opposite algebra $D^{\mathrm{op}}$ has the same maximal subfields as
$D$. So, unless $D \simeq D^{\mathrm{op}}$ (which is equivalent to
$D$ being of exponent 2), we have $\vert \gen(D) \vert > 1$. On the
other hand, as we already mentioned, it follows from (ABHN) that for
any quaternion algebra $D$ over a global field $K$ (and hence any central simple
$K$-algebra of exponent 2 over a global field is known to be
a quaternion algebra), $\gen(D)$ does reduce to
a single element. So, Question 1 really asks  about other fields
with this property. More specifically, we had asked earlier
if the field of rational functions $K = \Q(x)$ is such a
field. This question (in the context of quaternion algebras) was
answered in the affirmative by D.\,Saltman. Then, in \cite{GS},
Garibaldi and Saltman extended the result to  fields of the form $K
= k(x)$, where $k$ is any number field (and also to some other
situations). Recently, the following {\it Stability Theorem} was
proved in \cite{CRR2} for algebras of exponent 2 (the case of
quaternion algebras was considered earlier in \cite{RR}).
\begin{thm}\label{T:StabThm}
{\rm (\cite[Theorem 3.5]{CRR2})} Let $k$ be a field of
characteristic $\neq 2$. If $\vert \gen(D) \vert = 1$ for any
central division $k$-algebra $D$ of exponent 2 then the same
property holds for any central division algebra of exponent 2 over
the field of rational functions $k(x)$.
\end{thm}

\vskip2mm

\begin{cor}\label{C:1}
If $k$ is either a number field or a finite field of $\mathrm{char}
\neq 2$, and $K = k(x_1, \ldots , x_r)$ is a purely transcendental
extension then for any central division $K$-algebra $D$ of exponent
2 we have $\vert \gen(D) \vert = 1$.
\end{cor}

\vskip2mm

Furthermore, if $k$ is a field of $\mathrm{char} \neq 2$ such that
${}_2\Br(k) = 0$, then the field of rational functions $K = k(x_1,
\ldots , x_r)$ again satisfies the property described in Theorem
\ref{T:StabThm}. The existence of various examples in which the
genus of a division algebra of exponent 2 always reduces to one
element naturally leads to the question of {\it whether  the genus
of a quaternion algebra can \emph{ever} be nontrivial}. The answer
is `yes,' and the following construction of examples (described in
\cite[\S 2]{GS}) was offered by several people including Wadsworth,
Shacher, Rost, Saltman, Garibaldi ... We will describe only the
basic idea referring to \cite{GS} for the
details.

We start with two nonisomorphic quaternion division algebras $D_1$
and $D_2$ over a field $k$ of $\mathrm{char} \neq 2$ that have a
common quadratic subfield (e.g., one can take $k = \Q$ and
$\displaystyle D_1 = \left( \frac{-1 , 3}{\Q}\right)$ and
$\displaystyle D_2 = \left( \frac{-1 , 7}{\Q}\right)$). If $D_1$ and
$D_2$ already have the same quadratic subfields, we are done.
Otherwise, there exists a quadratic extension $k(\sqrt{d})$ that
embeds into $D_1$ but not into $D_2$. Then, using either properties
of quadratic forms or the ``index reduction formulas," one shows
that there exists an extension $k^{(1)}$ of $k$ (which is the field
of rational functions on a certain quadric) such that

\vskip2mm

\noindent $\bullet$ $k^{(1)}\otimes_k D_1$ and  $k^{(1)}\otimes_k D_2$ are
non-isomorphic division algebras over $k^{(1)}$, \ {\bf
but}

\vskip1mm

\noindent $\bullet$ $k^{(1)}(\sqrt{d})$ embeds into $k^{(1)}\otimes_k D_2$.

\vskip2mm

\noindent One deals with other subfields (in the algebras obtained
from $D_1$ and $D_2$ by applying the extension of scalars built at
the previous step of the construction), one at a time, in a similar
fashion. This process generates an ascending chain of fields
$$
k^{(1)} \subset k^{(2)} \subset k^{(3)} \subset \cdots ,
$$
and we let $K$ be the union (direct limit) of this chain. Then $K\otimes_k D_1$ and
$K\otimes_k D_2$ are non-isomorphic quaternion
division $K$-algebras having the same quadratic subfields; in
particular $\vert \gen(D_1 \otimes_k K) \vert > 1$. Note that the
resulting field $K$ has infinite transcendence degree over $k$,
hence is infinitely generated. Furthermore, some adaptation of the
above construction (cf.\,\cite{Meyer}) enables one to start with an
infinite sequence $D_1, D_2, D_3, \ldots$ of division algebras over
a field $k$ of characteristic $\neq 2$ that are pairwise
non-isomorphic but share a common quadratic subfield (e.g., one can
take $k = \Q$ and consider the family of algebras of the form
$\displaystyle \left(\frac{-1 , p}{\Q} \right)$ where $p$ is a prime
$\equiv 3(\mathrm{mod} \: 4)$), and then build an infinitely
generated field extension $K/k$ such that the algebras $D_i
\otimes_k K$ become pairwise non-isomorphic division algebras with
any two of them having the same quadratic subfields. This makes the
genus $\gen(D_1 \otimes_k K)$ infinite, and therefore brings us to
Question 2 of when one can guarantee the finiteness of the genus.
Here we have the following finiteness result.
\begin{thm}\label{T:Finite1}
{\rm (\cite[Theorem 3]{CRR1})} Let $K$ be a finitely generated
field. If $D$ is a central division $K$-algebra of exponent prime to
$\mathrm{char} \: K$, then $\gen(D)$ is finite.
\end{thm}

One of the questions about the genus of a division that remains open
after Theorems \ref{T:StabThm} and \ref{T:Finite1} is whether one
can find a quaternion division algebra {\it over a finitely
generated field} of characteristic $\neq 2$ with {\it nontrivial
genus}.

\vskip2mm

\noindent {\bf 4.12. The genus of an algebraic group.} We will now
discuss a possible generalization of the concept of the genus from
finite-dimensional central division algebras to arbitrary
absolutely almost simple algebraic groups obtained by replacing maximal
subfields with maximal tori.

So, let $G$ be an absolutely almost simple (simply connected or
adjoint) algebraic group over a field $K$. We define the genus
$\gen(G)$ to be the set of $K$-isomorphism classes of $K$-forms $G'$
of $G$ that have the same isomorphism classes of maximal $K$-tori as
$G$. Two remarks are in order. First, if $D$ is a finite-dimensional
central division algebra over a field $K$ and $G = SL_{1 , D}$ is
the corresponding group defined by elements of norm 1 in $D$, then
only maximal {\it separable} subfields of $D$ correspond to the
maximal $K$-tori of $G$. So, to avoid at least the obvious
discrepancies between the definitions of $\gen(D)$ and $\gen(G)$,
one should probably define the former in terms of {\it maximal
separable} subfields. We don't know however whether the definitions
of $\gen(D)$ in terms of all and only separable maximal subfields
would actually be distinct (these problems do not arise in Theorems
\ref{T:StabThm} and \ref{T:Finite1} as these treat only division
algebras whose degree is  prime to the characteristic of the
center). Second, one can give several alternative definitions of
$\gen(G)$ by working only with maximal {\it generic} $K$-tori, and
on the other hand by replacing $K$-isomorphisms of tori with
$K$-isogenies. It would be interesting to determine the precise
relationship between the various definitions; at this point, we will
just mention without  further elaboration that the definitions given
in terms of generic tori and $K$-isomorphism vs. $K$-isogeny {\it in
practice} lead to basically the same qualitative results (for the
reasons contained in Theorem \ref{T:Gen4} and the subsequent
discussion).

\vskip1mm

Building on Theorem \ref{T:Finite1}, we would like to propose the
following conjecture.

\vskip2mm

\addtocounter{thm}{2}

\noindent {\bf Conjecture 4.13.} {\it Let $G$ be an absolutely
almost simple (simply connected or adjoint) algebraic group over a
finitely generated field $K$. Assume that the characteristic of $K$
is either zero or is `good' for $G$. Then the genus $\gen(D)$ is
finite.}

\vskip2mm

The `bad' characteristics for each type are expected to be the
following:

\vskip1mm

$\bullet$ $\textsf{A}_{\ell}$\ --\ all prime divisors $p$ of $(\ell
+ 1)$, \ ${}^2\textsf{A}_{\ell}$\ --\ same primes and also $p = 2$;

\vskip.5mm

$\bullet$ $\textsf{B}_{\ell}$, $\textsf{C}_{\ell}$,
$\textsf{D}_{\ell}$\ --\ $p = 2$ (and possibly $p = 3$ for
${}^{3,6}\textsf{D}_4$)

\vskip.5mm

$\bullet$ $\textsf{E}_6, \textsf{E}_7, \textsf{E}_8, \textsf{F}_4$ and $\textsf{G}_2$\ --\ all prime
divisors of the order of the Weyl group.

\vskip1.5mm

\noindent The Isogeny Theorem \ref{T:Gen4} establishes some
connections between Conjectures 2.6 and 4.13, but more importantly,
we anticipate that the methods developed to deal with Conjecture
4.13 will be useful also in analyzing Conjecture 2.6 (in fact both
conjectures will be consolidated in \S \ref{S:Finite} into a single
conjecture - see Conjecture 5.4). Now, Theorem \ref{T:WC13}
confirms the conjecture in the situation where $K$ is a number
field. For general fields, the conjecture is known at this point
only for inner forms of type $\textsf{A}_{\ell}$.
\begin{thm}\label{T:Finite2}
{\rm (\cite[Theorem 5.3]{CRR2})} Let $G$ be a simply connected inner
form of type $\textsf{A}_{\ell}$ over a finitely generated field
$K$, and assume that the characteristic of $K$ is either zero or
does not divide $(\ell + 1)$. Then $\gen(G)$ is finite.
\end{thm}

The group $G$ in this theorem is of the form $\mathrm{SL}_{1 , A}$
for some central simple $K$-algebra $A$ of dimension $n^2$. While
every maximal $K$-torus of $G$ corresponds to some $n$-dimensional
commutative \'etale subalgebra of $A$ (and the same is true for any inner
$K$-form $G'$ of $G$), the existence of a $K$-isomorphism between
the tori a priori may not imply the existence of an isomorphism
between the \'etale algebras (it would be interesting to construct
such examples!). For this reason, Theorem \ref{T:Finite2} is {\it
not} an automatic consequence of Theorem \ref{T:Finite1}. The proof
of Theorem \ref{T:Finite2} uses generic tori, the isomorphisms
between which after appropriate scaling do extend to an isomorphism
between the \'etale algebras.

\vskip1mm

We mention in passing that there are other interesting approaches to
the definition of the genus. For example, Krashen and McKinnie
\cite{KMcK} defined the genus $\gen'(D)$ of a central division
$K$-algebra $D$ of prime degree based on all finite-dimensional
splitting fields. Furthermore, Merkurjev proposed to define the {\it
motivic genus} $\gen_m(G)$ of an absolutely almost simple algebraic
$K$-group $G$ along the lines suggested by Amitsur's Theorem, viz.
as the set of $K$-isomorphism classes of $K$-forms $G'$ such that
for {\it any} field extension $F/K$ the groups $G$ and $G'$ have the
same $F$-isomorphism classes of maximal $F$-tori. Since this concept
is less related to  weak commensurability, we will not discuss it
here referring the reader to \cite[Remark 5.6]{CRR2} for the details
(including an explanation of the term ``motivic").

\vskip5mm

\section{A finiteness result}\label{S:Finite}

The goal of this section is to try to establish a more direct
connection between the Finiteness Conjectures 2.6 and 4.13: while
such a connection undoubtedly exists, it has
manifested itself so far primarily through the fact that the
techniques developed for one of them are typically also useful for
the other, and not through any formal implications. We begin with a
new finiteness result over number fields which
%and then explain why it
implies the truth of {\it both} conjectures in this situation. We
then formulate and discuss a conjecture  which says that a similar statement
should be true over general fields (with some restrictions on the
characteristic).

\begin{thm}\label{T:Finite-NF}
Let $G$ be an absolutely almost simple
%(simply connected or adjoint)
algebraic group over a number field $K$, and let $\Gamma$ be a finitely generated
%torsion-free
Zariski-dense subgroup of $G(K)$ with  trace field $K$. Denote by $\gen(G ,
\Gamma)$  the set of isomorphism classes of $K$-forms $G'$ of $G$
having the following property: any \emph{generic} maximal $K$-torus
$T$ of $G$ that contains an element of $\Gamma$ of infinite order is isogenous to
some maximal $K$-torus $T'$ of $G'$. Then $\gen(G , \Gamma)$ is
finite.
\end{thm}
\begin{proof}
We begin with a  statement which is valid over any finitely
generated field $K$ of characteristic zero as its proof relies only
on the facts established in this generality.
\begin{lemma}\label{L:5-1}
Let $G' \in \gen(G , \Gamma)$, and let $L$ (resp., $L'$) denote the
minimal Galois extension of $K$ over which $G$ (resp., $G'$) is of inner type, i.e.,
is  an inner form of a split group. Then $L = L'$, and hence $G'$ is an inner form of $G$
over $K$.
\end{lemma}
\begin{proof}
According to \cite{PR-Reg} (cf.\,also \cite[Theorem 9.6]{PR-Gen}),
there exists a regular semi-simple element $\gamma \in \Gamma$ of
infinite order such that the torus $T :=
Z_{G}(\gamma)^{\circ}$ is generic over $\mathcal{L} := LL'.$ By our
assumption, there exists a maximal $K$-torus $T'$ of $G'$ for which
there is a $K$-isogeny $\nu \colon T \to T'$. We then have
the following commutative diagram:
$$
\xymatrix{& {\mathrm{GL}}(X(T) \otimes_{\Z} \Q) \ar[dd]_{\tilde{\nu}}  \\ \Ga(\overline{K}/K) \ar[ru]^{\theta_T} \ar[rd]_{\theta_{T'}} & \\ & {\mathrm{GL}}(X(T') \otimes_{\Z} \Q),}
$$
where $\overline{K}$ is an algebraic closure of $K$ and
$\tilde{\nu}$ is the isomorphism induced by $\nu$. We note that for
any field extension $F$ of $K$ contained in $\overline{K}$, the map
$\tilde{\nu}$ gives an isomorphism between the images of
$\Ga(\overline{K}/F)$ under $\theta_T$ and $\theta_{T'}$, hence
\begin{equation}\label{E:5-1}
\vert \theta_T(\Ga(\overline{K}/F)) \vert = \vert
\theta_{T'}(\Ga(\overline{K}/F)) \vert.
\end{equation}
Furthermore, since both $G$ and $G'$ are of inner type over
$\mathcal{L}$, we have
\begin{equation}\label{E:5-2}
\theta_T(\Ga(\overline{K}/\mathcal{L})) = W(G , T) \ \ \ \text{and}
\ \ \ \theta_{T'}(\Ga(\overline{K}/\mathcal{L})) = W(G' , T')
\end{equation}
(cf.\,\cite[Lemma 4.1]{PR1}).

Now, assume that $L' \not\subset L$, i.e. $L \subsetneqq
\mathcal{L}$. Again, since $G$ is of inner type over $L$, we have
\begin{equation}\label{E:5-3}
\theta_T(\Ga(\overline{K}/L)) = W(G , T).
\end{equation}
On the other hand, it follows from (\ref{E:5-2}) that
$\theta_{T'}(\Ga(\overline{K}/L))$ contains $W(G' , T')$ but in fact
is strictly larger as by our assumption $G'$ is {\it not} of inner
type over $L$ (cf.\,\cite[Lemma 4.1]{PR1}). Thus,
$$
\vert \theta_{T'}(\Ga(\overline{K}/L)) \vert > \vert W(G' , T')
\vert =\vert W(G,T)\vert=\vert\theta_T(\mathrm{Gal}(\overline{K}/L))\vert,
$$
which contradicts (\ref{E:5-1})
for $F= L$. Similarly, the assumption $L \not\subset L'$ would
imply that
$$
\vert \theta_T(\Ga(\overline{K}/L')) \vert > \vert
\theta_{T'}(\Ga(\overline{K}/L') \vert,
$$
contradicting (\ref{E:5-1}) for $F = L'$.
\end{proof}

\vskip2mm

By \cite[Theorem 6.7]{PlR}, one can find a finite subset $S_1
\subset V^K$  such that $G$ is quasi-split over $K_v$ for any $v \in
V^K \setminus S_1$. Now, let $\pi \colon \widetilde{G} \to G$ be the
universal cover. It follows from \cite{Weis} that one
can find a finite subset $S_2 \subset V^K$ containing $V_{\infty}^K$
such that $\Gamma \subset G(\cO(S_2))$ and there exists a subgroup $\Delta$ of  $\pi^{-1}(\Gamma)$ of finite index
contained in
$\widetilde{G}(\cO(S_2))$ whose closure in the group of $S_2$-adeles
$\widetilde{G}(\A_{S_2})$ is open. Set $S = S_1 \cup S_2$.

\vskip1mm

\begin{lemma}\label{L:5-2}
Every $G' \in \gen(G , \Gamma)$ is quasi-split over $K_v$
for $v \in V^K \setminus S$.
\end{lemma}
\begin{proof}
We first make the following general observation. Let $\mathscr{G}_0$
be a quasi-split semi-simple group over a field $\cK$, and let
$\mathscr{G}$ be an {\it inner} $\cK$-form of $\mathscr{G}_0$; then
the fact that
\begin{equation}\label{E:5-4}
\mathrm{rk}_{\cK}\: \mathscr{G} \geqslant \mathrm{rk}_{\cK}\:
\mathscr{G}_0
\end{equation}
implies that $\mathscr{G}$ is itself quasi-split (and hence in
(\ref{E:5-4}) we actually have  equality). Indeed, since
$\mathscr{G}$ is an inner twist of $\mathscr{G}_0$, the $*$-actions
on the Tits indices of $\mathscr{G}_0$ and $\mathscr{G}$ are
identical (cf.\,\cite[Lemma 4.1(a)]{PR1}). Since the $\cK$-rank of a
semi-simple group equals the number of distinguished orbits under
the $*$-action in its Tits index, and all the orbits in the Tits
index of $\mathscr{G}_0$ are distinguished as the latter is
$\cK$-quasi-split, condition (\ref{E:5-4}) implies that the same is
true for $\mathscr{G}$, making it quasi-split.

Returning to the proof of the lemma, let us fix $v \in V^K \setminus
S$, and set $\cK = K_v$. By Lemma \ref{L:5-1}, the group $G'$ is an
inner form of $G$ over $K$, and hence over $\cK$. It follows from
the construction of $S$ that the closure of $\Delta$ in
$\widetilde{G}(\cK)$ is open, and then so is the closure of $\Gamma$
in $G(\cK)$. Using Theorem 3.4 in \cite{PR-Fields}, we find a
regular semi-simple element $\gamma \in \Gamma$ of infinite order such that the torus
$T = Z_{G}(\gamma)^{\circ}$ is generic over $K$ and contains a
maximal $\cK$-split torus of $G$, i.e. $\mathrm{rk}_{\cK}\: T =
\mathrm{rk}_{\cK}\: G$. By our assumption, $T$ is $K$-isogenous to a
maximal $K$-torus $T'$ of $G'$. Then we have
$$
\mathrm{rk}_{\cK}\: G' \geqslant \mathrm{rk}_{\cK}\: T' =
\mathrm{rk}_{\cK} \: T = \mathrm{rk}_{\cK}\: G.
$$
Since by construction $\mathscr{G}_0 := G$ is quasi-split over
$\cK$, applying the remark following (\ref{E:5-4}) to $\mathscr{G}
:= G'$, we obtain that $G'$ is quasi-split over $\cK$, as
required.
\end{proof}

\vskip2mm

Now, let $G_0$ be the quasi-split inner $K$-form of $G$. Fix an arbitrary $G' \in \gen(G ,
\Gamma)$. It follows from Lemma \ref{L:5-1} that $G'$ is an inner $K$-form of $G$
and so it is an inner $K$-form of $G_0$. Hence $G'$ is obtained by twisting $G_0$ by a class $\zeta \in H^1(K, G_0)$.
This class  lies in
$$
\Sigma_S := \mathrm{Ker}\Big( H^1(K , G_0) \longrightarrow
\bigoplus_{v \in V^K \setminus S} H^1(K_v , G_0) \Big).
$$
since for $v\notin S$, $G'$, being quasi-split over $K_v$, is
$K_v$-isomorphic to $G_0$. (For this one needs to observe that the
map $H^1(F , G_0) \to H^1(F , \mathrm{Aut}\: G_0)$ has trivial
kernel for any field extension $F/K$ which follows from the fact
that $\mathrm{Aut}\: G_0$ is a semi-direct product of $G_0$ and a
$K$-subgroup of symmetries of the Dynkin diagram.) However,
$\Sigma_S$ is known to be finite for any finite subset $S \subset
V^K$ (cf.\,\cite[Ch.\,III, \S 4, Theorem 7]{Serre}), and the
finiteness of $\gen(G , \Gamma)$ follows.\end{proof}

It is easy to see that Theorem \ref{T:Finite-NF} implies the truth
of both Conjectures 2.6 and 4.13 over number fields - the connection
with Conjecture 4.13 is obvious while in order to connect with
Conjecture 2.6 one needs to use the Isogeny Theorem \ref{T:Gen4}.
Moreover, this kind of implication would remain valid over a general
field, and we would like to end this section with the following
conjecture that suggests a uniform approach to Conjectures 2.6 and
4.13.

\vskip2mm

\noindent {\bf Conjecture 5.4.} {\it  Let $G$ be an absolutely
almost simple  algebraic group over a field $K$ of ``good"
characteristic, and let $\Gamma$ of  $G(K)$ be a finitely
generated Zariski-dense subgroup with  field of
definition\footnote{In characteristic zero the field of definition
coincides with the trace field by Vinberg's theorem \cite{Vin1}; in
positive characteristic, particularly in characteristics 2 and 3,
the notion of the ``right" field of definition is more tricky - see
\cite{Pink}, but we will not get into these details here.} $K$. Let
$\gen(G , \Gamma)$  be the set of isomorphism classes of $K$-forms
$G'$ of $G$ having the following property: any maximal generic
$K$-torus $T$ of $G$ that contains an element of $\Gamma$ of infinite order
is $K$-isogenous to some maximal $K$-torus $T'$ of $G'$. Then
$\gen(G , \Gamma)$ is finite.}

%there exists a maximal $K$-torus $T'$ of $G'$
%for which there is a $K$-isomorphism (resp., $K$-isogeny) $T \to
%T'$. Then  $\gen(G , \Gamma)$ is finite.} \footnotetext{The finite

\vskip5mm

\section{Back to geometry}\label{S:Geom}
\vskip2mm

\noindent {\bf 6.1. Locally symmetric spaces.} Let $G$ be a
connected adjoint semi-simple real agebraic group, let $\cG = G(\R)$
considered as a real Lie group, and let $\fX = \mathcal{K}
\backslash \cG$, where $\mathcal{K}$ is a maximal compact subgroup
of $\cG$, be the associated symmetric space endowed with the
Riemannian metric coming from the Killing form on the Lie algebra of
$\cG$. Furthermore, given a discrete torsion-free subgroup $\Gamma$
of $\cG$, we let $\fX_{\Gamma} := \fX/\Gamma$ denote the
corresponding locally symmetric space. We say that $\fX_{\Gamma}$ is
{\it arithmetically defined} if the subgroup $\Gamma \subset G(\R)$
is arithmetic in the sense of  \S 3.1. Finally, we recall that $\Gamma$
is called a {\it lattice} if $\fX_{\Gamma}$ (or equivalently
$\cG/\Gamma$) has finite volume. As in \S~1, we let
$L(\fX_{\Gamma})$ denote the (weak) length spectrum of
$\fX_{\Gamma}$.

\vskip2mm

Now, given  two simple algebraic $\R$-groups $G_1$ and $G_2$, and
discrete torsion-free subgroups  $\Gamma_i \subset \cG_i = G_i(\R)$
for $i = 1, 2$, we will denote the corresponding locally symmetric
spaces by $\fX_{\Gamma_i}$. The following statement establishes a
connection between the length-commensurability of $\fX_{\Gamma_1}$
and $\fX_{\Gamma_2}$ (i.e. the condition $\Q \cdot L(\fX_{\Gamma_1})
= \Q \cdot L(\fX_{\Gamma_2})$) and the weak commensurability of
$\Gamma_1$ and $\Gamma_2$.

\addtocounter{thm}{1}

\begin{thm}\label{T:6-1}
{\rm (\cite[Corollary 2.8]{PR-Gen})} Assume that $\Gamma_i$ is a
\emph{lattice} in $\cG_i$. If the locally symmetric spaces
$\fX_{\Gamma_1}$ and $\fX_{\Gamma_2}$ are length-commensurable, then
the subgroups $\Gamma_1$ and $\Gamma_2$ are weakly commensurable.
\end{thm}

Some remarks are in order. As for Riemann surfaces (cf.\,\S 1.2),
closed geodesics in $\fX_{\Gamma}$ correspond to (nontrivial)
semi-simple elements of $\Gamma$, but in the general case the
equation for the length is significantly more complicated: instead
of just the logarithm of an eigenvalue, we get basically the square
root of a sum of squares of the logarithms of certain eigenvalues
(see \cite[Proposition 8.5]{PR1} for the precise formula), although
for lattices in simple groups not isogenous to ${\mathrm{SL}}_2(\R)$ these
eigenvalues are actually algebraic numbers. Unfortunately, at this
point there are no results in transcendental number theory that
would enable one to analyze  expressions of this kind - most
available results are for {\it linear} forms in terms of logarithms
of algebraic numbers (cf.\,\cite{Baker}). This forced us to base our
analysis of the lengths of closed geodesics on a conjecture in
transcendental number theory, known as Schanuel's conjecture, which
is widely believed to be true but has been proven so far in very few
situations (for the reader's convenience, we recall its statement
below). The use of this conjecture is essential in the case of
locally symmetric spaces of rank $> 1$, making our geometric results
in this case {\it conditional on Schanuel's conjecture}. At the same
time, the results for rank one spaces apart from the following
exceptional case

\vskip2.5mm

$(\mathcal{E})$: \parbox[t]{15cm}{$G_1 = \mathrm{PGL}_2$ and
$\Gamma_1$ cannot be conjugated into $\mathrm{PGL}_2(K)$ for any
number field $K \subset \R$ \newline while $G_2 \neq
\mathrm{PGL}_2$,}

\vskip3mm

\noindent rely only on the Gel'fond-Schneider Theorem (as a
replacement of Schanuel's conjecture), hence are {\it
unconditional}. Besides, a statement similar to Theorem \ref{T:6-1}
(and in fact more precise) can be proven under weaker conditions  on
 $\Gamma_1$ and $\Gamma_2$ - see \cite[Theorem
2.7]{PR-Gen}. A detailed discussion of these issues is contained in
\cite[\S 2]{PR-Gen} and will not be repeated here. So, we conclude
simply by recalling the statement of Schanuel's conjecture.

\vskip2mm

\noindent {\bf 6.3. Schanuel's conjecture.} {\it If $z_1, \ldots ,
z_n \in \C$ are linearly independent over $\Q$, then the
transcendence degree (over $\Q$) of the field generated by
$$
z_1, \ldots , z_n; \ e^{z_1}, \ldots , e^{z_n}
$$
is $\geqslant n$.}

\vskip2mm

In fact, we will only need the following consequence of this conjecture:
for nonzero algebraic numbers $a_1, \ldots , a_n$, (any values of)
their logarithms $\log a_1, \ldots , \log a_n$ are algebraically
independent once they are linearly independent (over $\Q$).
%In order to apply this statement in our situation, we first prove
%the following elementary lemma.

\vskip2mm

\addtocounter{thm}{1}

\vskip2mm

Theorem \ref{T:6-1} enables us to ``translate" the algebraic results
from \S\S 2-3 about weakly commensurable Zariski-dense subgroups
into the geometric setting. In particular, applying Theorems
\ref{T:WC1} and \ref{T:WC2} we obtain the following.
\begin{thm}\label{T:WC8}
Let $G_1$ and $G_2$ be connected absolutely simple real algebraic
groups, and let $\fX_{\Gamma_i}$ be a locally symmetric space of
finite volume, of $\cG_i,:=G_i(\R)$ for $i = 1,\,2.$ If $\fX_{\Gamma_1}$ and
$\fX_{\Gamma_2}$ are length-commensurable, then {\rm{(i)}} either
$G_1$ and $G_2$ are of same Killing-Cartan type, or one of them is
of type $\textsf{B}_{n}$ and the other is of type $\textsf{C}_{n}$
for some $n \geqslant 3$, {\rm{(ii)}} $K_{\Gamma_1} = K_{\Gamma_2}.$
\end{thm}

It should be pointed out that assuming Schanuel's conjecture in all
cases, one can prove this theorem (in fact, a much stronger
statement -- see \cite[Theorem 1]{PR-Fields} and \cite[Theorem
8.1]{PR-Gen}) assuming only that $\Gamma_1$ and $\Gamma_2$ are
finitely generated and Zariski-dense.

\vskip1mm

Next, using Theorems \ref{T:WC4} and \ref{T:WC5} we obtain
\begin{thm}\label{T:WC9}
Let $G_1$ and $G_2$ be connected absolutely simple real algebraic
groups, and let $\cG_i = G_i(\R),$ for $i = 1,\,2.$ Then the set of
arithmetically defined locally symmetric spaces $\fX_{\Gamma_2}$ of
$\cG_2$, which are length-commensurable to a given arithmetically
defined locally symmetric space $\fX_{\Gamma_1}$ of $\cG_1$, is a
union of finitely many commensurability classes. In fact, it
consists of a single commensurability class if $G_1$ and $G_2$ have
the same type different from $\textsf{A}_{n}$, $\textsf{D}_{2n+1}$,
with  $n
> 1$, or $\textsf{E}_6.$
\end{thm}

\vskip2mm

Furthermore, Theorems \ref{T:WC6}  and \ref{T:WC7} imply the
following rather surprising result which has so far defied all
attempts to find a purely geometric proof.
\begin{thm}\label{T:WC10}
Let $G_1$ and $G_2$ be connected absolutely simple real algebraic
groups, and let $\fX_{\Gamma_1}$ and $\fX_{\Gamma_2}$ be
length-commensurable locally symmetric spaces of $\cG_1$ and
$\cG_2$,  respectively, of finite volume. Assume that at least one
of the spaces is arithmetically defined. Then the other space is
also arithmetically defined, and the compactness of one of the
spaces implies the compactness of the other.
\end{thm}

In fact, if one of the spaces is compact and the other is not, the
weak length spectra $L(\fX_{\Gamma_1})$ and $L(\fX_{\Gamma_2})$ are
quite different -- see \cite[Theorem 5]{PR-Fields} and \cite[Theorem
8.6]{PR-Gen} for a precise statement (we note that the proof of this
result uses Schanuel's conjecture in all cases).

\vskip2mm

Finally, we will describe some applications to isospectral compact
locally symmetric spaces. So, in the remainder of this section, the
locally symmetric spaces $\fX_{\Gamma_1}$ and $\fX_{\Gamma_2}$ as
above will be assumed to be {\it compact}. Then, as we discussed in
\S 1, the fact that $\fX_{\Gamma_1}$ and $\fX_{\Gamma_2}$ are
isospectral implies that $L(\fX_{\Gamma_1}) = L(\fX_{\Gamma_2})$, so
we can use our results on length-commensurable spaces. Thus,  in
particular we obtain the following.
\begin{thm}\label{T:WC11}
If $\fX_{\Gamma_1}$ and $\fX_{\Gamma_2}$ are isospectral, and
$\Gamma_1$ is arithmetic, then so is $\Gamma_2$.
\end{thm}

(Thus, the spectrum of the Laplace-Beltrami operator can see if the fundamental group is
arithmetic or not -- to our knowledge, no results of this kind,
particularly for general locally symmetric spaces, were previously
known in spectral theory.)

\vskip2mm

The following theorem settles the question ``Can one hear the shape
of a drum?'' for arithmetically defined compact locally symmetric
spaces.
\begin{thm}\label{T:WC12}
Let $\fX_{\Gamma_1}$ and $\fX_{\Gamma_2}$ be compact locally
symmetric spaces associated with absolutely simple real algebraic
groups $G_1$ and $G_2$, and assume that at least one of the spaces
is arithmetically defined. If $\fX_{\Gamma_1}$ and $\fX_{\Gamma_2}$
are isospectral then $G_1 = G_2 := G$. Moreover, unless $G$ is of
type $\textsf{A}_{n}$, $\textsf{D}_{2n+1}$ $(n > 1)$, or
$\textsf{E}_6$, the spaces $\fX_{\Gamma_1}$ and $\fX_{\Gamma_2}$ are
commensurable.
\end{thm}

It should be noted that our methods based on length-commensurability
or weak commensurability leave room for the following ambiguity in
Theorem \ref{T:WC12}: either $G_1 = G_2$ or $G_1$ and $G_2$ are
$\R$-split forms  of types $\textsf{B}_{n}$ and $\textsf{C}_{n}$ for
some $n \geqslant 3$ - and this ambiguity is unavoidable, cf.\,\cite[Theorem 4]{PR-Fields} and the end of \S 7 in \cite{PR-Gen}.
The fact that in the latter case the locally symmetric spaces cannot
be isospectral was shown by Sai-Kee Yeung \cite{SKY} by comparing
the traces of the heat operator (without using Schanuel's
conjecture), which leads to the statement of the theorem given
above.

\vskip5mm

\noindent {\small {\bf Acknowledgements.} Both authors were
partially supported by the NSF (grants DMS-1001748, DMS-0965758 and
DMS-1301800) and the Humboldt Foundation. The first-named author
thanks the Institute for Advanced Study (Princeton) for its
hospitality and support during 2012. The second-named author thanks
Shing-Tung Yau and Lizhen Ji for the invitation to the GAAGTA
conference. He is also grateful to the Mathematics Department at
Yale for the friendly atmosphere and support during his visit in
2013.}

\vskip5mm

\bibliographystyle{amsplain}

\end{document}